\documentclass[11pt, reqno]{amsart}

\usepackage{latexsym}
\usepackage{amssymb}
\usepackage{mathrsfs}
\usepackage{amsmath}
\usepackage{fancybox,color}
\usepackage{enumerate}
\usepackage[latin1]{inputenc}
\usepackage{eurosym}

\newcommand{\kom}[1]{}
\renewcommand{\kom}[1]{{\bf [#1]}}

 \def\1{\raisebox{2pt}{\rm{$\chi$}}}

\newtheorem{theorem}{Theorem}[section]

\newtheorem{lemma}[theorem]{Lemma}

\newtheorem{definition}[theorem]{Definition}
\newtheorem{remark}[theorem]{Remark}

\newcommand{\R}{{\mathbb R}}

\newcommand{\N}{{\mathbb N}}
\newcommand{\Z}{{\mathbb Z}}

 \newcommand{\eps}{{\varepsilon}}
 \def\1{\raisebox{2pt}{\rm{$\chi$}}}
 
\newcommand{\abs}[1]{\left|#1\right|}

\newcommand{\osc}{\operatorname{osc}}

\def\vint_#1{\mathchoice%
          {\mathop{\kern 0.2em\vrule width 0.6em height 0.69678ex depth -0.58065ex
                  \kern -0.8em \intop}\nolimits_{\kern -0.4em#1}}%
          {\mathop{\kern 0.1em\vrule width 0.5em height 0.69678ex depth -0.60387ex
                  \kern -0.6em \intop}\nolimits_{#1}}%
          {\mathop{\kern 0.1em\vrule width 0.5em height 0.69678ex
              depth -0.60387ex
                  \kern -0.6em \intop}\nolimits_{#1}}%
          {\mathop{\kern 0.1em\vrule width 0.5em height 0.69678ex depth -0.60387ex
                  \kern -0.6em \intop}\nolimits_{#1}}}
\def\vintslides_#1{\mathchoice%
          {\mathop{\kern 0.1em\vrule width 0.5em height 0.697ex depth -0.581ex
                  \kern -0.6em \intop}\nolimits_{\kern -0.4em#1}}%
          {\mathop{\kern 0.1em\vrule width 0.3em height 0.697ex depth -0.604ex
                  \kern -0.4em \intop}\nolimits_{#1}}%
          {\mathop{\kern 0.1em\vrule width 0.3em height 0.697ex depth -0.604ex
                  \kern -0.4em \intop}\nolimits_{#1}}%
          {\mathop{\kern 0.1em\vrule width 0.3em height 0.697ex depth -0.604ex
                  \kern -0.4em \intop}\nolimits_{#1}}}

\newcommand{\kint}{\vint}

\newcommand{\aveint}[2]{\mathchoice%
          {\mathop{\kern 0.2em\vrule width 0.6em height 0.69678ex depth -0.58065ex
                  \kern -0.8em \intop}\nolimits_{\kern -0.45em#1}^{#2}}%
          {\mathop{\kern 0.1em\vrule width 0.5em height 0.69678ex depth -0.60387ex
                  \kern -0.6em \intop}\nolimits_{#1}^{#2}}%
          {\mathop{\kern 0.1em\vrule width 0.5em height 0.69678ex depth -0.60387ex
                  \kern -0.6em \intop}\nolimits_{#1}^{#2}}%
          {\mathop{\kern 0.1em\vrule width 0.5em height 0.69678ex depth -0.60387ex
                  \kern -0.6em \intop}\nolimits_{#1}^{#2}}}
\newcommand{\ud}{\, d}

\newcommand{\ol}{\overline}
\newcommand{\Om}{\Omega}

\newcommand{\dist}{\operatorname{dist}}

\newcommand{\vp}{\varphi}

\newcommand{\tr}{\operatorname{tr}}

\begin{document}

\title[Regularity for time-dependent tug-of-war games]{Local regularity for time-dependent tug-of-war games with varying probabilities}

\author[Parviainen]{Mikko Parviainen}
\address{Department of Mathematics and Statistics, University of
Jyvaskyla, PO~Box~35, FI-40014 Jyvaskyla, Finland}
\email{mikko.j.parviainen@jyu.fi}

\author[Ruosteenoja]{Eero Ruosteenoja}
\address{Department of Mathematics and Statistics, University of
Jyvaskyla, PO~Box~35, FI-40014 Jyvaskyla, Finland}
\email{eero.ruosteenoja@jyu.fi}

\date{\today}

\begin{abstract} 
We study local regularity properties of value functions of time-dependent tug-of-war games. For games with constant probabilities we get local Lipschitz continuity. For more general games with probabilities depending on space and time we obtain H\"older and Harnack estimates. The games have a connection to the normalized $p(x,t)$-parabolic equation $(n+p(x,t))u_t=\Delta u+(p(x,t)-2) \Delta_{\infty}^N u$.
\end{abstract}
\maketitle

\section{Introduction}
\label{sec:intro}
In this paper we study local regularity properties of tug-of-war type games related to parabolic PDEs. First, we establish asymptotic Lipschitz continuity for the value functions of the game with constant probabilities, and then continue analyzing the regularity of a more general game with space and time dependent probabilities that we call $p(x,t)$-game. 

The value functions of this particular two player zero sum game satisfy the so called dynamic programming principle (hereafter DPP)
\[
\begin{split}
u_\eps(x,t)=&\frac{\alpha(x,t)}{2} (\sup_{B_{\eps}(x)} u(y,t-\frac{\eps^2}{2})+\inf_{B_{\eps}(x)} u(y,t-\frac{\eps^2}{2}))\\
&+\beta(x,t) \kint_{B_{\eps}(x)} u(y,t-\frac{\eps^2}{2}) \ud y,
\end{split}
\]
 which may arise for example from stochastic games or discretization schemes. In game terms, this equation can be heuristically interpreted as summing up the three different alternatives of the game round with the corresponding $(x,t)$-dependent probabilities while  the step takes $\frac{\eps^2}{2}$ time.

Lipschitz estimate for the game with constant probabilities is based on the good symmetry properties produced by utilizing cancellation strategies that allows us to directly obtain Lipschitz continuity. 
In the $p(x,t)$-case the symmetry properties and sharp cancellation effects break down. Moreover, global approaches to the problem are hampered by the of loss of translation invariance, which makes it hard to keep track of accumulated error. 

Our proofs are of local nature. The idea of the proof for H\"older continuity of $p(x,t)$-game arises from the stochastic game theory: we start the game simultaneously at two points $x$ and $z$ and try to pull the points closer to each other, where 'closer' is in terms of a suitable comparison function. In particular, in the stochastic terminology the process is a supermartingale.
 To show that we may pull the points closer in this sense, we may consider the process in the higher dimensional space by setting  $(x,z)\in \R^{2n}$ and then apply suitable strategies for such a game. There are several differences in the  parabolic setting compared to the elliptic proofs in  \cite{luirops13} and \cite{luirop} related to controlling the dynamic effects. Indeed, in the Lipschitz proof we utilize estimates for probability distributions on different time instances whereas in the elliptic case it suffices to deal with the long time limit distribution. In the case of the $p(x,t)$-game the resulting DPP is in $\R^{2n+1}$, and comparison functions will have to take the time direction into account. 
 
As an application, our results can be used to prove local Lipschitz continuity for the solutions to the normalized $p$-parabolic equation
\[
\begin{split}
(n+p)u_t=\abs{\nabla u}^{2-p}\text{div}(\abs{\nabla u}^{p-2} \nabla u)=\Delta u+(p-2)\Delta_{\infty}^N u,
\end{split}
\] 
where $\Delta_{\infty}^N u=\langle D^2 u \frac{\nabla u}{\abs{\nabla u}}, \frac{\nabla u}{\abs{\nabla u}}\rangle$ is the normalized or game theoretic infinity Laplacian.
This equation has been recently studied by Jin and Silvestre \cite{jins15}, Banerjee and Garofalo \cite{banerjeeg13},  Does \cite{does11}, as well as Manfredi, Parviainen and Rossi \cite{manfredipr10}. In the $p(x,t)$-case, we show that under suitable assumptions the value functions of the game converge to the unique viscosity solution of the Dirichlet problem to the normalized $p(x,t)$-parabolic equation
\[
\begin{split}
(n+p(x,t))u_t=\Delta u+(p(x,t)-2) \Delta_{\infty}^N u.
\end{split}
\]
 However, a priori our methods and results are not relying on the PDE techniques; rather they are quite different from those. 

The connection between the infinity Laplacian and tug-of-war games was established by Peres, Schramm, Sheffield and Wilson in \cite{peresssw09}, for the $p$-Laplacian in \cite{peress08} and for the normalized $p$-parabolic equation by Manfredi, Parviainen and Rossi in \cite{manfredipr10}, see also \cite{banerjeeg15b}.

This paper is organized as follows. In Section \ref{prel} we fix the notation and define the game. In Section \ref{constant} we assume that $p(x,t)\equiv p>2$ is a constant and obtain local asymptotic Lipschitz continuity for value functions by using game-theoretic methods. In Section \ref{Holder,Harnack} we get H\"older and Harnack estimates for the $p(x,t)$-game. In Section \ref{uniform} it is proved that value functions of the $p(x,t)$-game converge uniformly to a viscosity solution of the normalized $p(x,t)$-parabolic equation. In Section \ref{uniq} we show that there is a unique viscosity solution to the $p(x,t)$-parabolic equation with given continuous boundary data.

\noindent \textbf{Acknowledgements.} MP was supported by the Academy of Finland. ER was supported by the Vilho, Kalle and Yrj\"o V\"ais\"al\"a foundation.

\section{Preliminaries}\label{prel}
Throughout the paper $\Omega\subset \R^n$ is a bounded domain. If not mentioned otherwise, for $T>0$, $\Omega_T:=\Omega\times(0,T)$ is a parabolic cylinder with the parabolic boundary  
\[
\Gamma_T:=\{\partial \Omega\times [0,T]\}\cup \{\Omega\times \{0\}\}.
\]
For our game we also need the parabolic boundary strip of width $\eps>0$,
\[
\Gamma^\eps_T:=\left(\Gamma_\eps\times (-\frac{\eps^2}{2},T]\right)\cup \left(\Omega\times(-\frac{\eps^2}{2},0] \right).
\]
Here
\[
\Gamma_\eps:=\{x\in \R^n\setminus \Omega\ :\ \dist(x,\partial \Omega)\leq \eps\}
\]
is the $\eps$-boundary strip of $\Omega$.

For a measurable function $p:\Omega_T\rightarrow (2,\infty)$, we define the functions $\alpha:\Omega_T\rightarrow (0,1)$ and $\beta:\Omega_T\rightarrow (0,1)$,
\[
\alpha(x,t)=\frac{p(x,t)-2}{p(x,t)+n},\ \ \beta(x,t)=\frac{n+2}{p(x,t)+n}.
\]
Notice that $\alpha(x,t)+\beta(x,t)=1$ for all $(x,t)\in \Omega_T$. 


Next we define a tug-of-war type game, which we call $p(x,t)$-\emph{game} to emphasize the connection with $p(x,t)$-Laplacian, see Section \ref{uniform}. The game is a zero sum stochastic game between Player I and Player II in $\Omega_T$. Fix $\eps>0$. First a token is placed at $(x_0,t_0)\in \Omega_T$. With probability $\alpha(x_0,t_0)$, the players flip a fair coin, and the winner of the toss moves the token to a point
\[
(x_1,t_1)\in B_\eps(x_0)\times \{t_0-\frac{\eps^2}{2}\},
\]
according to his or her strategy. We use a notation $B_\eps(x_0)$ for an open ball centered at $x_0$ with radius $\eps$. With probability $\beta(x_0,t_0)$, the token moves according to the uniform probability to a random point $(x_1,t_1)$ in a set $B_\eps(x_0)\times \{t_0-\frac{\eps^2}{2}\}$, and in this paper we call such moves \emph{random vectors} for short. From $(x_1,t_1)$ the game continues according to the same rules, and the token moves to a point 
\[
(x_2,t_2)\in B_\eps(x_1)\times \{t_1-\frac{\eps^2}{2}\}.
\]
We denote by $(x_\tau,t_\tau)\in \Gamma^\eps_T$ the first point of the sequence on $\Gamma^\eps_T$. Then Player II pays Player I the payoff $F(x_\tau,t_\tau)$, where $F:\Gamma^\eps_T\rightarrow [-M,M]$ is a given measurable payoff function. Naturally, Player I tries to maximize the payoff and Player II tries to minimize it. The number of steps during the game is bounded,
\[
\tau\leq 2\eps^{-2}t_0+1\leq 2\eps^{-2}T+1.
\]
The value function $u_\eps$ of the game is 
\[
u_\eps(x_0,t_0)=\sup_{S_\text{I}} \inf_{S_\text{II}}\mathbb{E}^{(x_0,t_0)}_{S_\text{I},S_\text{II}}[F(x_\tau,t-\frac{\eps^2}{2}\tau)],
\]
where $S_\text{I}$ and $S_\text{II}$ are strategies of Player I and Player II. For further details on stochastic vocabulary regarding tug-of-war games, we refer to \cite{peresssw09}. 

Since the number of steps during the game is bounded, adding a bounded running payoff to the game would not cause any new difficulties. In the case of unlimited number of steps the situation is different, see \cite{ruosteenoja16}.  

A crucial property of value functions of tug-of-war type games is DPP characterization. In the parabolic case this characterization is much easier to verify than in the elliptic case. Moreover, proving DPP characterization for value functions of our game does not differ from the case where the probabilities $\alpha$ and $\beta$ are space independent. The following two lemmas can be proved by using the techniques of \cite{luirops14}. We use the notation
\[
\vint_{B_r} u\, \mathrm{d}x:=\frac{1}{|B_r|}\int_{B_r} u\, \mathrm{d}x
\]
for the mean value of a function $u$ in a ball $B_r$. Here $|B_r|$ denotes the Lebesgue measure of $B_r$.
\begin{lemma}
For given $\eps>0$ and payoff function $F$ on $\Gamma^\eps_T$, there is a unique measurable function $u$ equal to $F$ on $\Gamma^\eps_T$ and satisfying the parabolic DPP
\[
\begin{split}
u(x,t)=&\frac{\alpha(x,t)}{2} (\sup_{B_{\eps}(x)} u(y,t-\frac{\eps^2}{2})+\inf_{B_{\eps}(x)} u(y,t-\frac{\eps^2}{2}))\\
&+\beta(x,t) \kint_{B_{\eps}(x)} u(y,t-\frac{\eps^2}{2}) \ud y
\end{split}
\]
for $(x,t)\in \Omega_T$.
\end{lemma}
 
\begin{lemma}
Given $\eps>0$ and a bounded payoff function $F$ on $\Gamma^\eps_T$, the value function $u_\eps$ satisfies the parabolic DPP.
\end{lemma}

A typical idea to estimate the value function $u_\eps$ is to fix a strategy for one of the players. We may also localize the situation by using a new stopping time $\tau^*\leq \tau$. The following lemma is a standard tool for fixed strategies. Again we omit the proof which is similar to \cite[Lemma 2.3]{ruosteenoja16}.

\begin{lemma}\label{fix.strategia} If the game starts from $(x_0,t_0)\in \Omega_T$ and $\tau^*<2t_0\eps^{-2}$ is a stopping time, then
\[
u_\eps(x_0,t_0)\geq \inf_{S_\text{\emph{II}}}\mathbb{E}^{(x_0,t_0)}_{S^0_\text{\emph{I}},S_\text{\emph{II}}}u_\eps(x_{\tau^*},t_0-\frac{\tau^*}{2}\eps^2)
\] 
for any fixed strategy $S^0_\text{\emph{I}}$ of Player I, and
\[
u_\eps(x_0,t_0)\leq \sup_{S_\text{\emph{I}}}\mathbb{E}^{(x_0,t_0)}_{S_\text{\emph{I}},S^0_\text{\emph{II}}}u_\eps(x_{\tau^*},t_0-\frac{\tau^*}{2}\eps^2)
\]
for any fixed strategy $S^0_\text{\emph{II}}$ of Player II.
\end{lemma}


\section{Lipschitz estimate for $p$-game}\label{constant} In this section we study local regularity properties of $p(x,t)$-game when $p(x,t)\equiv p>2$ is a constant. Then the probability functions are also constants, $\alpha(x,t)\equiv \alpha\in (0,1)$ and $\beta(x,t)\equiv \beta\in (0,1)$. This game was defined in \cite{manfredipr10}.

We start with constant $p$ for simplicity: these games have symmetry properties suitable for cancellation strategy idea, developed in \cite{luirops13}, to get asymptotic Lipschitz continuity. In order to establish this, we use the following stochastic estimate, which combines well known Hoeffding's and Kolmogorov's inequalities. 

\begin{lemma}\label{lemma 3.1}
Consider i.i.d.\ symmetric real-valued random variables $Y_m$, $m=1,...,N,$ for which $|Y_m|\leq b$ for some $b>0$. Then for $\lambda>0$ the following inequalities hold:
\[
\mathbb{P}(|Y_1+...+Y_N|\geq \lambda)\leq 2 \exp\left(-\frac{\lambda^2}{2Nb^2} \right),
\]
\[
\mathbb{P}(\max_{1\leq m\leq N}|Y_1+...+Y_m|\geq \lambda)\leq 2 \mathbb{P}(|Y_1+...+Y_N|\geq \lambda).
\]
\end{lemma}

When the game starts from $(x_0,t_0)\in \Omega_T$ and Player I follows \emph{a cancellation strategy with target $z$}, she tries to cancel the earliest move of Player I which she has not yet canceled. If there are no moves to cancel, she tries to pull the token to the direction of vector $(z-x_0)\in \Omega$. Notice that Player I pays no attention to random moves. 

We want to use the cancellation strategy to prove asymptotic Lipschitz estimate for the $p(x,t)$-game with constant $p$. The two main difficulties are the possibility to reach the maximum number of steps too soon and the case of different time levels. We estimate the probability for reaching maximum number of steps in the proof of Theorem \ref{Liplemma}, and the problem of different time levels is solved in Theorem \ref{aika-Lip}.

\begin{theorem}\label{Liplemma} 
Suppose that $B_{6r}(z_0)\subset \Omega$, where $0<\eps<r<\left(\frac{\alpha T}{6}\right)^\frac12$. Then, for points $(x,t),(y,t)\in B_r(z_0)\times (\frac{6r^2}{\alpha},T)\subset \Omega_T$ and for sufficiently small $\eps$, the value function $u_\eps$ satisfies the Lipschitz estimate 
\[
\left|u_\eps(x,t)-u_\eps(y,t)\right|\leq C(p,n)\frac{\left|x-y\right|}{r}\left\|u_\eps\right\|_\infty+C'(p,n)\frac\eps r\left\|u_\eps\right\|_\infty.
\] 
\end{theorem}

\begin{proof}
Because of the error term, we may suppose that $|x-y|\geq \eps$. Let $z$ be the midpoint of $[x,y]\subset \Omega$ and suppose first that 
\[
u_\eps(y,t)\geq u_\eps(x,t).
\]
When the game starts from $(x,t)=:(x_0,t_0)$, Player II follows the cancellation strategy with a target $z$. Let us define the stopping time $\tau^*$. There are four conditions to stop the game: 
\begin{enumerate} \item
Player II wins $\lceil \left|x-z\right|/\eps \rceil$ fair coin tosses more than Player I. \item Player I wins $\lceil r/\eps \rceil$ fair coin tosses more than Player II. \item The sum of random vectors has length larger than $r$. \item We reach the maximum number of steps.
\end{enumerate}
When the game starts from $(y,t)=:(y_0,t_0)$, Player I follows the cancellation strategy trying to pull towards $z$, and we define $\tau^*$ as before by changing the roles of the players. By using the cancellation effect and Lemma \ref{fix.strategia}, we obtain
\[
\left|u_\eps(x_0,t_0)-u_\eps(y_0,t_0)\right|\leq  2\left\|u_\eps\right\|_\infty\sum^k_{j=1}\overline{P}_j+2\delta\left\|u_\eps\right\|_\infty,
\]
where $\overline{P}_j$ is the probability that $\tau^*=j$ and the game ended because of Condition 2 or 3, and $\delta$ is the probability that the game ended because the maximum number of steps was reached. The number $k$ is the maximum number of steps during the game, $k=\lceil 2\eps^{-2}t_0 \rceil$.

We get an upper estimate for $\sum \overline{P}_j$ from \cite[Lemma 3.1]{luirops13}. The lemma gives an upper bound $C(p,n)|x-y|/r$ for the probability $P'$ that the tug-of-war with noise ends because of Condition 2 or 3. Since there is not Condition 4 in the elliptic case (there is not an upper bound for the number of steps during the game), we have
\[
\sum^k_{j=1}\overline{P}_j\leq P'\leq C(p,n)\frac{|x-y|}{r}.
\] 
Hence, we get
\begin{equation}\label{lipeq1}
\left|u_\eps(x,t)-u_\eps(y,t)\right|\leq C(p,n)\frac{\left|x-y\right|}{r}\left\|u_\eps\right\|_\infty+2\delta \left\|u_\eps\right\|_\infty.
\end{equation}

The previous inequality also holds if $u_\eps(x,t)>u_\eps(y,t)$, which can be seen by fixing a cancellation strategy for Player I when starting from $(x,t)$ and for Player II when starting from $(y,t)$. 

The main part of this proof is to estimate the probability $\delta$ that the game ends when the maximum number $\lceil 2t_0/\eps^2\rceil$ of steps is reached. First we need a rough estimate for the number of fair coin tosses between the players during the game. Denote by $Z_m$ the Bernoulli variables with $Z_m\in \{0,1\}$ and $\mathbb{P}(Z_m=1)=\alpha$. Define
\[
A:=\left\{\sum^{l}_{m=1}Z_m>\frac\alpha2 l\ \text{for all}\ l\geq \eps^{-1}\right\}.
\]
We estimate 
\begin{align*}
\mathbb{P}(A^c)&=\mathbb{P}\left(\sum^{l}_{m=1}Z_m\leq \frac\alpha2 l\ \text{for some}\ l\geq \eps^{-1} \right)\\
& \leq \sum_{l\geq \eps^{-1}}\mathbb{P}\left(\sum^{l}_{m=1}Z_m\leq \frac\alpha2 l\right)\\
& \leq \sum_{l\geq \eps^{-1}}\mathbb{P}\left(\left|\sum^{l}_{m=1}Z_m-l\alpha \right|\geq \frac\alpha2 l\right)\\
& = \sum_{l\geq \eps^{-1}}\mathbb{P}\left(\left|\sum^{l}_{m=1}(Z_m-\alpha) \right|\geq \frac\alpha2 l\right).
\end{align*}
Using Lemma \ref{lemma 3.1} with $Y_m=Z_m-\alpha$, $\lambda=\frac\alpha2 l$, $b=1$ and $N=l$ gives
\[
\sum_{l\geq \eps^{-1}}\mathbb{P}\left(\left|\sum^{l}_{m=1}(Z_m-\alpha) \right|\geq \frac\alpha2 l\right)\leq \sum_{l\geq \eps^{-1}}2\  \text{exp}(-\frac{\alpha^2}{8}l)\leq O(\eps).
\]
Hence, for small enough $\eps$ there is a constant $C'(p,n)>0$ such that 
\begin{equation}\label{prob.A}
\mathbb{P}(A)\geq 1-C'(p,n)\frac{\eps}{r}.
\end{equation}

Supposing that $\lceil \frac\alpha2 \eps^{-2} t_0\rceil$ is an even number, we estimate combinatorially the probability $\widetilde{P}_0$ that after exactly $\lceil \frac\alpha2 \eps^{-2} t_0\rceil$ fair coin flips there have been exactly the same number of heads and tails, 
\begin{align*}
\widetilde{P}_0&=\binom{\lceil \frac\alpha2 \eps^{-2} t_0\rceil}{\frac12\lceil \frac\alpha2 \eps^{-2} t_0\rceil}\left(\frac12\right)^{\lceil \frac\alpha2 \eps^{-2} t_0\rceil}\\
&=\frac12 \frac34 \frac56 ... \frac{\lceil \frac\alpha2 \eps^{-2} t_0\rceil-1}{\lceil \frac\alpha2 \eps^{-2} t_0\rceil}\\
&\leq \left(\frac12 \frac23 \frac34 ... \frac{\lceil \frac\alpha2 \eps^{-2} t_0\rceil}{\lceil \frac\alpha2 \eps^{-2} t_0\rceil+1}\right)^{\frac12}\\
& =\left(\frac{1}{\lceil \frac\alpha2 \eps^{-2} t_0\rceil+1}\right)^{\frac12}\\
& \leq \frac{\eps}{3r},
\end{align*} 
where in the last inequality we used the requirement $t_0>\frac{6r^2}{\alpha}$. For probability $\widetilde{P}_k$, $k\in \Z$, that after  $\lceil \frac\alpha2 \eps^{-2} t_0\rceil$ of fair coin flips there have been $k$ heads more than tails, we have $\widetilde{P}_k\leq \widetilde{P}_0$. (When $k$ is negative, $\widetilde{P}_k$ means that there have been $-k$ tails more than heads.) We get the estimate
\[
\sum^{\lceil |x-y|/\eps\rceil}_{k=-\lceil |x-y|/\eps\rceil}\widetilde{P}_k\leq \left(\frac{2|x-y|}{\eps}+1\right)\widetilde{P}_0\leq \frac{|x-y|}{r}. 
\] 
Denote by D an event that the event $A$ occurred and at the time $\lceil \frac\alpha2 \eps^{-2} t_0\rceil$ of fair coin flips there have been at least $\lceil \frac{|x-y|}{\eps}\rceil$ heads more than tails. Moreover, denote by E an event that the event $A$ occurred and there have been at least $\lceil \frac{|x-y|}{\eps}\rceil$ heads more than tails at some point before $\lceil \frac\alpha2 \eps^{-2}t_0\rceil$ fair coin flips. By the previous estimate we have
\begin{align*}
\mathbb{P}(D)&\geq \frac12 \left(1-\frac{|x-y|}{r}\right)(1-C'(p,n)\frac{\eps}{r})\\
& \geq \frac12 \left(1-2C'(p,n)\frac{|x-y|}{r}\right).
\end{align*}
To estimate $\mathbb{P}(E)$, observe first that
\[
\mathbb{P}(E\cap D)=\frac12\mathbb{P}(E)
\]
and
\[
\mathbb{P}(E^c\cap D)\leq \frac{\eps}{3r}.
\]
Since
\[
\mathbb{P}(D)=\mathbb{P}(E\cap D)+\mathbb{P}(E^c\cap D),
\]
we get
\[
\mathbb{P}(E)\geq 1-2C'(p,n)\frac{|x-y|}{r}-\frac{\eps}{3r}.
\]
Since the probability that the game ends before step $\lceil 2t_0/\eps^2\rceil$ is greater than $\mathbb{P}(E)$, we get an estimate for $\delta$,
\[
\delta\leq C(p,n)\frac{|x-y|}{r}+3C'(p,n)\frac\eps r,
\]
and recalling estimate \eqref{lipeq1}, we have
\begin{align*}
\left|u_\eps(x,t)-u_\eps(y,t)\right| &\leq C(p,n)\frac{\left|x-y\right|}{r}\left\|u_\eps\right\|_\infty+2\delta \left\|u_\eps\right\|_\infty\\
& \leq 2C(p,n)\frac{\left|x-y\right|}{r}\left\|u_\eps\right\|_\infty+6C'(p,n)\frac\eps r \left\|u_\eps\right\|_\infty.\qedhere
\end{align*}
\end{proof}

\begin{theorem}\label{aika-Lip}
Let $x,y\in \Omega$ and $t=:t_0$ satisfy the conditions of Theorem \ref{Liplemma} and $t_1\in (t_0,T)$ satisfy $t_1-t_0\leq r^2$. Then for $(x, t_1),(y,t_0)\in \Omega_T$ we have the Lipschitz estimate
\begin{align*}
\left|u_\eps(x,t_1)-u_\eps(y,t_0)\right|\leq C(p,n)\frac{\left|x-y\right|+\left|t_1-t_0\right|^{\frac12}}{r}\left\|u_\eps\right\|_\infty\\
\phantom{{}=C}+C'(p,n)\frac{\eps^\frac12}{r}\left\|u_\eps\right\|_\infty.\notag
\end{align*}
\end{theorem}

\begin{proof}
We prove the case $x=y$, for otherwise we use triangle inequality and Theorem \ref{Liplemma}. Because of the error term, we may suppose that $t_1\geq t_0+\eps^2$. Denote 
\[
s:=\sqrt{t_1-t_0}\geq \eps.
\]
Suppose first that $u_\eps(y,t_1)\geq u_\eps(y,t_0)$. The game starts from $(y,t_1)$. Player II uses a strategy $S^0_{\text{II}}$ in which he pulls towards $y$ and stays there if possible. The game ends when the token leaves the cylinder $S:=B_s(y)\times (t_0,t_1)$ for the first time. Let A be the event that the token hits the bottom of $S$. Then, regardless of the strategy of Player I,
\[
P:=\mathbb{P}(A)\geq \left(\frac{1}{10}\right)^{2(n+1)^2}.
\]
This estimate follows from the proof of Lemma \ref{pos.laajennus} below. 

Denote
\[
M:=C(p,n)\frac{s}{r}||u_\eps||_\infty,
\]
where $C(p,n)$ is the constant from Theorem \ref{Liplemma}. By using Theorem \ref{Liplemma} to estimate values of $u_\eps$ in the ball $B_s(y)$, we get
\begin{align*}
u_\eps(y,t_1)-u_\eps(y,t_0)&\leq P(u_\eps(y,t_0)+M)+(1-P)\sup_{\partial B_s(y)\times [t_0,t_1]}u_\eps-u_\eps(y,t_0)\\
& =PM+(1-P)(\sup_{\partial B_r(y)\times [t_0,t_1]}u_\eps-u_\eps(y,t_0))\\
&\leq PM+(1-P)(\sup_{t\in [t_0,t_1]}u_\eps(y,t)+M-u_\eps(y,t_0))\\
& =M+(1-P)(\sup_{t\in [t_0,t_1]}u_\eps(y,t)-u_\eps(y,t_0)). 
\end{align*}
Choose $k\in \N$ such that 
\[
(1-P)^k\leq C(p,n)\frac{s}{r}. 
\]
By continuing the previous estimation we get
\begin{align*}
u_\eps(y,t_1)-u_\eps(y,t_0)&\leq M\sum^{k-1}_{j=0}(1-P)^j+(1-P)^k(\sup_{t\in [t_0,t_1]}u_\eps(y,t)-u_\eps(y,t_0))\\
&\leq \frac{1}{P}M+2C(p,n)\frac{s}{r}||u_\eps||_\infty\\
&=\widetilde{C}(p,n)\frac{s}{r}||u_\eps||_\infty.
\end{align*}
If $u_\eps(y,t_1)< u_\eps(y,t_0)$, we fix a strategy for Player I when starting from $(y,t_1)$ and by symmetric argument we get
\[
u_\eps(y,t_0)-u_\eps(y,t_1)\leq \widetilde{C}(p,n)\frac{s}{r}||u_\eps||_\infty.
\]
Hence, we have 
\begin{align*}
|u_\eps(y,t_1)-u_\epsilon(y,t_1)|&\leq \widetilde{C}(p,n)\frac{s}{r}||u_\eps||_\infty\\
&= \widetilde{C}(p,n)\frac{|t_1-t_0|^{\frac12}}{r}||u_\eps||_\infty.
\end{align*}
The error term of the scale $\eps^{1/2}$ has to be added when $t_1-t_0\leq \eps^2$.
\end{proof}

\section{H\"older and Harnack estimates for $p(x,t)$-game}\label{Holder,Harnack}
In this Section we study regularity properties of the parabolic $p(x,t)$-game, which was defined in section \ref{prel}. We assume throughout the section that $u_\eps>0$ is a value function of the game. In the first subsection we show asymptotic H\"older continuity for $u_\eps$, and then continue with Harnack's inequality in the second subsection.

\subsection{Asymptotic H\"older continuity} Since our location dependent parabolic game is not translation invariant, we cannot immediately use the cancellation strategy. Instead, we use a more general idea developed by Luiro and Parviainen for the elliptic case in \cite{luirop}. The main idea is to start the game simultaneously at two points and try to pull them closer to each other by using a suitable comparison function $f$ with a certain favorable curvature in space. The player trying to pull the two points closer, say Player I, has a certain flexibility in her strategy depending on what the opponent does. If Player II does not pull the points further away from each other, then Player I tries to pull them directly closer. Instead, if Player II tries to pull the points almost optimally further away, then Player I aims at the exactly opposite step. 

As in the previous section concerning Lipschitz regularity, we break the proof of parabolic H\"older continuity into two parts. In the first part, Theorem \ref{Holder1}, we consider the case where the points $x,y\in \Omega$ are at the same time level $t$ in $\Omega_T$. We use the strategy of \cite{luirop}, but add a time-dependent term $g(t)=|t|^{\delta/2}$ to the comparison function $f$. The purpose of the term $g$ in our comparison function $F(x,t)=f(x)+g(t)$ is to get the right boundary values for $F$ without allowing too large error in estimates.

In the other part of the proof of H\"older continuity we handle the time direction. This part is easier, and we could actually prove it by utilizing the technique we used in the proof of Theorem \ref{aika-Lip}. However, we present another proof relying more on the DPP property of the value function $u_\eps$.

\begin{theorem}\label{Holder1}
Let $B_{2r}(0)\times [-2r^2,-\frac12 r^2]\subset \Omega\times (-T,T)$. Then the value function $u_\eps$ satisfies the following H\"older estimate for some $\delta=\delta(n)\in (0,1)$,
\[
|u_\eps(x,t)-u_\eps(y,t)|\leq C(n)\frac{|x-y|^\delta}{r^\delta}\left\|u_\eps\right\|_\infty+C'(n)\frac{\eps^\delta}{r^\delta}\left\|u_\eps\right\|_\infty, 
\]
when $x,y\in B_r(0)$ and $t\in (-r^2,-\frac12 r^2)$.
\end{theorem}

\begin{proof}
Denote
\[
S_1:=B_r(0)\times (-r^2,-\frac12 r^2),\ S_2:=B_{2r}(0)\times (-2r^2,-\frac12 r^2).
\]
To define a suitable comparison function, define the functions $g$, $f_1$ and $f_2$,
\[
g(t)=|t|^{\delta/2},
\]
\[
f_1(x,z)=C(n)|x-z|^\delta +|x+z|^2,
\]
as well as
\[ f_2(x,z) = \left\{ 
  \begin{array}{l l}
    C^{2(N-i)}\eps^\delta & \quad \text{if}\ (x,z)\in A_i,\\
    0 & \quad \text{if}\ |x-z|>N\frac{\eps}{10}.
  \end{array} \right.\]
Here
\[
A_i:=\{(x,z)\in \R^{2n}\ :\ (i-1)\frac{\eps}{10}<|x-z|\leq i\frac{\eps}{10}\}
\]
for $i=\{1,...,N\}$. Finally, our comparison function is
\[
F(x,z,t)=f(x,z)+g(t),
\]
where
\[
f(x,z)=f_1(x,z)-f_2(x,z).
\]
We use this notation to emphasize the time dependent term $g$ needed in the parabolic case. 

By scaling, we may assume that
\[
0\leq u_\eps\leq r^\delta\ \text{in}\ S_2\setminus S_1.
\]
This implies
\[
u_\eps(x,t)-u_\eps(z,t)-F(x,z,t)\leq C^{2N}\eps^\delta\quad \text{in}\ S_2\setminus S_1,
\]
and we want to show that the same inequality holds in $S_1$. Suppose not. Then 
\[
M:=\sup_{(x',t'),(z',t')\in S_1}(u_\eps(x',t')-u_\eps(z',t')-F(x',z',t'))>C^{2N}\eps^\delta.
\]
Thriving for contradiction, let $\eta>0$ and choose $(x,t), (z,t)\in S_1$ such that
\begin{equation}\label{M arvio}
u_\eps(x,t)-u_\eps(z,t)-F(x,z,t)\geq M-\eta.
\end{equation}
Recall that DPP for $u_\eps$ reads as
\begin{align*}
u_\eps(x,t)&=\frac{\alpha(x,t)}{2}\left\{\sup_{y\in B_\eps(x)}u_\eps(y,t-\frac{\eps^2}{2})+\inf_{y\in B_\eps(x)}u_\eps(y,t-\frac{\eps^2}{2}) \right\}\\
&\phantom{{}=\frac{\alpha(x,t)}{2}}+\beta(x,t)\vint_{B_\eps(x)}u_\eps(y,t-\frac{\eps^2}{2})dy.
\end{align*}
By using the DPP characterization for the difference $u_\eps(x,t)-u_\eps(z,t)$, we see that
\[
u_\eps(x,t)-u_\eps(z,t)=I_1+I_2+I_3,
\]
where
\begin{align*}
I_1&=\frac{\alpha(z,t)}{2}\bigg(\sup_{B_\eps(x)}u_\eps(y,t-\frac{\eps^2}{2})-\inf_{B_\eps(z)}u_\eps(y,t-\frac{\eps^2}{2})\\
&\phantom{{}=\frac{\alpha(z,t)}{2}}+\inf_{B_\eps(x)}u_\eps(y,t-\frac{\eps^2}{2})-\sup_{B_\eps(z)}u_\eps(y,t-\frac{\eps^2}{2})\bigg),
\end{align*}
\[
I_2=\beta(x,t)\left(\vint_{B_\eps(x)}u(y,t-\frac{\eps^2}{2})dy-\vint_{B_\eps(z)}u(y,t-\frac{\eps^2}{2})dy\right),
\]
and
\begin{align*}
I_3=&\frac{\alpha(x,t)-\alpha(z,t)}{2}\\
& \left(\sup_{B_\eps(x)}u_\eps(y,t-\frac{\eps^2}{2})+\inf_{B_\eps(x)}u_\eps(y,t-\frac{\eps^2}{2})-2\vint_{B_\eps(z)}u(y,t-\frac{\eps^2}{2})dy\right).
\end{align*}
This identity together with inequality \eqref{M arvio} gives
\begin{equation}\label{contradiction}
M\leq I_1+I_2+I_3-F(x,z,t)+\eta.
\end{equation}

We are going to estimate the terms $I_1$, $I_2$ and $I_3$ to get a contradiction with \eqref{contradiction}. To be more precise, we are going to show the following inequalities,
\[
\alpha(z,t)M>I_1-\alpha(z,t)(F(x,z,t)-\eta),
\]
\[
\beta(x,t)M> I_2-\beta(x,t)((F(x,z,t)-\eta),
\]
as well as
\[
(\alpha(x,t)-\alpha(z,t))M>I_3-(\alpha(x,t)-\alpha(z,t))(F(x,z,t)-\eta).
\]
To estimate $I_1$, first we prove the following inequalities
\[
\sup_{B_\eps(x)}u_\eps(y,t-\frac{\eps^2}{2})-\inf_{B_\eps(z)}u_\eps(y,t-\frac{\eps^2}{2})\leq M+\sup_{B_\eps(x)\times B_\eps(z)}F(\ol x,\ol z,t-\frac{\eps^2}{2})+\eta
\]
and
\[
\inf_{B_\eps(x)}u_\eps(y,t-\frac{\eps^2}{2})-\sup_{B_\eps(z)}u_\eps(y,t-\frac{\eps^2}{2})\leq M+\inf_{B_\eps(x)\times B_\eps(z)}F(\ol x,\ol z,t-\frac{\eps^2}{2})+\eta.
\]
The first inequality follows by picking $x'\in B_\eps(x)$, $z'\in B_\eps(z)$ such that $u_\eps(x')\geq \sup_{B_\eps(x)}u_\eps-\eta/2$ and $u_\eps(z')\leq \inf_{B_\eps(z)}u_\eps-\eta/2$ and estimating
\begin{align*}
\sup_{y\in B_\eps(x)}& u_\eps(y,t-\frac{\eps^2}{2})-\inf_{B_\eps(z)}u_\eps(y,t-\frac{\eps^2}{2})\\
&\leq u_\eps(x',t-\frac{\eps^2}{2})-u_\eps(z',t-\frac{\eps^2}{2})+\eta\\
&\leq M+F(x',z',t-\frac{\eps^2}{2})+\eta\\
&\leq M+\sup_{(y,y')\in B_\eps(x)\times B_\eps(z)}F(y,y',t-\frac{\eps^2}{2})+\eta.
\end{align*}
The second inequality follows the same way, and we get an estimate for $I_1$,
\begin{align*}
&I_1-\frac{\alpha(z,t)}{2}\eta\\
&\leq \alpha(z,t)\left(M+\frac12 \left(\sup_{B_\eps(x)\times B_\eps(z)}F(x,z,t-\frac{\eps^2}{2})+\inf_{B_\eps(x)\times B_\eps(z)}F(x,z,t-\frac{\eps^2}{2})\right)\right).
\end{align*}
Let us show that
\begin{align*}
F(x,z,t)&>\frac12 \left(\sup_{B_\eps(x)\times B_\eps(z)}F(x',z',t-\frac{\eps^2}{2})+\inf_{B_\eps(x)\times B_\eps(z)}F(x',z',t-\frac{\eps^2}{2})\right)+2\eta\\
& =\frac12 \left(\sup_{B_\eps(x)\times B_\eps(z)}f(x',z')+\inf_{B_\eps(x)\times B_\eps(z)}f(x',z')\right)+\left|t-\frac{\eps^2}{2}\right|^{\delta/2}+\eta.
\end{align*}
Since $t<-\frac{r^2}{2}$ and $r<1$, we get an estimate
\begin{align}\label{aika-arvio}
\left|t-\frac{\eps^2}{2}\right|^{\delta/2}-|t|^{\delta/2}&\leq \left(\frac{r^2}{2}+\frac{\eps^2}{2}\right)-\left(\frac{r^2}{2}\right)^{\delta/2}\nonumber\\
& \leq \left(\frac{r^2}{2}\right)^{\delta/2}\left(1+\frac{\eps^2}{r^2}\right)^{\delta/2}-\left(\frac{r^2}{2}\right)^{\delta/2}\nonumber\\
& \leq \left(\frac{r^2}{2}\right)^{\delta/2}\left(1+\frac{\eps^2}{r^2}\right)-\left(\frac{r^2}{2}\right)^{\delta/2}\nonumber\\
& =\frac{\eps^2}{r^2}\leq r^{\delta-2}\eps^2.
\end{align}
Hence, it suffices to show that
\[
f(x,z)>\frac12 \left(\sup_{B_\eps(x)\times B_\eps(z)}f+\inf_{B_\eps(x)\times B_\eps(z)}f\right)+r^{\delta-2}\eps^2.
\]

Throughout the proof the error caused by the term $g$ is in the acceptable scale $r^{\delta-2}\eps^2$.
  
During the rest of the argument we just write $\sup f$ and $\inf f$ meaning that $\sup$ and $\inf$ are taken over $B_\eps(x)\times B_\eps(z)$. 

Suppose first that $|x-z|>N\frac{\eps}{10}$. Then $f_2=0$.  Choose $h_x,h_z\in B_\eps(0)$ such that
\[
\sup f_1\leq f_1(x+h_x,z+h_z)+\eta.
\]
Let $\theta=\frac{1}{10}$ and assume first that
\[
(h_x-h_z)^2_V\geq (4-\theta)\eps^2,
\]
where $V$ is the space spanned by $x-z$ and
\[
(h_x-h_z)_V=(h_x-h_z)\cdot \frac{x-z}{|x-z|}.
\]
To estimate $\sup f_1+\inf f_2-2f_1$, it is useful to write Taylor's expansion for $f_1(x+h_x,z+h_z)$ as
\begin{align*}
f_1&(x+h_x,z+h_z)\\
& =f_1(x,z)+C\delta |x-z|^{\delta-1}(h_x-h_z)_V+2(x+z)\cdot(h_x+h_z)\\
& +\frac C2\delta|x-z|^{\delta-2}\left((\delta-1)(h_x-h_z)^2_V+(h_x-h_z)^2_{V^\bot}\right)\\
& +|h_x+h_z|^2+\mathcal{E}_{x,z}(h_x,h_z).
\end{align*}
Here $\mathcal{E}_{x,z}$ is an error term satisfying
\begin{align*}
\mathcal{E}_{x,z}(h_x,h_z)&\leq C|(h_x,h_z)|^3(|x-z|-2\eps)^{\delta-3}\\
&\leq 10\eps^2|x-z|^{\delta-2}
\end{align*}
when $N$ is large enough, for example $N>100C/\delta$.

By using the Taylor estimate and the estimate for the error term, we obtain 
\begin{align*}
&\sup f_1+\inf f_2-2f_1\\
&\leq f_1(x+h_x,z+h_z)+f_1(x-h_x,z-h_z)-2f_1(x,z)+\eta\\
&= \frac C2 \delta|x-z|^{\delta-2}\left(2(\delta-1)(h_x-h_z)^2_V+2(h_x-h_z)^2_{V^\bot}\right)\\
& +2|h_x+h_z|^2+\mathcal{E}_{x,z}(h_x,h_z)+\mathcal{E}_{x,z}(-h_x,-h_z)+\eta\\
&\leq |x-z|^{\delta-2}(20-C\delta)\eps^2+8\eps^2+\eta\\
& \leq -\widetilde{C}r^{\delta-2}\eps^2+8\eps^2+\eta<-2r^{\delta-2}\eps^2,
\end{align*}
when $\widetilde{C}=C\delta-20$ has been chosen large.

If
\[
(h_x-h_z)^2_V< (4-\theta)\eps^2,
\]
then
\[
(h_x-h_z)_V\leq (2-\theta/4)\eps,
\]
and the second order term of the Taylor estimate, together with the error term, can be estimated from above by
\begin{align*}
\frac C2& \delta|x-z|^{\delta-2}(2\eps)^2+(2\eps)^2+10\eps^2|x-z|^{\delta-2}\\
& \leq 10C\delta|x-z|^{\delta-2}\eps^2.
\end{align*}
Now we get
\begin{align*}
&\sup f_1+\inf f_1-2f_1\\
&\leq f_1(x+h_x,z+h_z)+f_1(x-\eps\frac{x-z}{|x-z|},x+\eps\frac{x-z}{|x-z|})-2f_1(x,z)+\eta\\
&\leq C\delta|x-z|^{\delta-1}(-\theta\eps/4)+16\eps+10C\delta|x-z|^{\delta-2}\eps^2+\eta\\
&\leq 10\frac{C}{\delta}C\delta|x-z|^{\delta-2}(-\theta/5)\eps^2+10C\delta|x-z|^{\delta-2}\eps^2+\eta\\
& \leq -\frac{1}{10}C^2r^{\delta-2}\eps^2\\
& <-2r^{\delta-2}\eps^2,
\end{align*}
when $C$ is large enough.

Suppose next that $|x-z|\leq N\frac{\eps}{10}$. Then a straightforward estimate gives 
\[
|f_1(x+h_x,z+h_z)-f_1(x,z)|\leq 3C\eps^\delta.
\]
We also have
\[
\inf (f_1-f_2)\leq \sup f_1-10C\eps^\delta-2f_2,
\]
which implies
\begin{align*}
\sup f +\inf f &\leq 2\sup f_1-10C\eps^\delta-2f_2(x,z)\\
&\leq 2f_1+6C\eps^\delta-10C\eps^\delta-2f_2+\eps^\delta\\
& \leq 2f-2\eps^\delta
\end{align*}
when $C$ is large enough. Since $-\eps^\delta\leq -|x-z|^{\delta-2}\eps^2\leq -r^{\delta-2}\eps^2$, It follows that
\[
f(x,z)>\frac12 \left(\sup_{B_\eps(x)\times B_\eps(z)}f+\inf_{B_\eps(x)\times B_\eps(z)}f\right)+r^{\delta-2}\eps^2.
\]
Hence, we have shown that 
\[
I_1-\frac{\alpha(z,t)}{2}\delta<\alpha(z,t)(M+F(x,z,t)),
\]
or equivalently,
\[
\alpha(z,t)M>I_1-\alpha(z,t)F(x,z,t)-\frac{\alpha(z,t)}{2}\delta.
\]

Let us next estimate $I_2$. We want to show that
\[
\beta(x,t)M> I_2-\beta(x,t)((F(x,z,t)+\eta),
\]
where
\[
I_2=\beta(x,t)\left(\vint_{B_\eps(x)}u(y,t-\frac{\eps^2}{2})dy-\vint_{B_\eps(z)}u(y,t-\frac{\eps^2}{2})dy\right).
\]
Let $P_{x,z}(h)$ be a mirror point of $h$ with respect to $\text{span}(x-z)^\perp$. If $|x-z|\geq 2\eps$, we get an estimate
\begin{align*}
I_2&=\frac{\beta(x,t)}{|B_\eps|}\bigg(\int_{B_\eps(0)}u_\eps(x+h)-u_\eps(z+P_{x,z}(h))-F(x+h,z+P_{x,z}(h))dy\\
&\phantom{{}=\int_{B_\eps(0)}u_\eps(x+h)}+\int_{B_\eps(0)}F(x+h,z+P_{x,z}(h))dy\bigg)\notag\\
&\leq \beta(x,t)M+\frac{\beta(x,t)}{|B_\eps|}\int_{B_\eps(0)}F(x+h,z+P_{x,z}(h))dy.
\end{align*}
If $|x-z|\leq 2\eps$, there is a perfect cancellation in the intersection $B_\eps(x)\cap B_\eps(z)$. We refer to \cite{luirop} for details and just state that in this case we have an estimate
\[
I_2\leq \beta(x,t)M+\beta(x,t)J_1,
\]
where
\[
J_1=\frac{1}{|B_\eps|}\left(\int_{B_\eps(0)\setminus B_\eps(x-z)}F(x+h,z+P_{x,z}(h))dh+\int_{B_\eps(x)\cap B_\eps(z)}F(y,y)dy\right).
\]
We want to show that
\[
F(x,z,t)>J_1.
\]
Notice that since
\begin{align*}
&\int_{B_\eps(0)\setminus B_\eps(x-z)}F(x+h,z+P_{x,z}(h))dh+\int_{B_\eps(x)\cap B_\eps(z)}F(y,y)dy\\
&\leq \int_{B_\eps(0)\setminus B_\eps(x-z)}f(x+h,z+P_{x,z}(h))dh+\int_{B_\eps(x)\cap B_\eps(z)}f(y,y)dy\\
&\phantom{{}=\int_{B_\eps(0)\setminus B_\eps(x-z)}}+\frac{2}{|B_\eps|}\int_{B_\eps}\left|t-\frac{\eps^2}{2}\right|^{\delta/2}dy,
\end{align*}
it is sufficient to show that 
\begin{align*}
&f(x,z,t)-2r^{\delta-2}\eps^2\\
&>\frac{1}{|B_\eps|}\left(\int_{B_\eps(0)\setminus B_\eps(x-z)}f(x+h,z+P_{x,z}(h))dh+\int_{B_\eps(x)\cap B_\eps(z)}f(y,y)dy\right).
\end{align*}
If $|x-z|>N\frac{\eps}{10}$, the key estimate is
\[
\eps^2|x-z|^{\delta-2}(10-\frac{C\delta}{4(n+2)})+2r^{\delta-2}\eps^2<0,
\]
which holds when $C$ is sufficiently large. In the same manner, if $|x-z|\leq N\frac{\eps}{10}$, the additional error term $2r^{\delta-2}\eps^2$ does not cause extra difficulty compared to the elliptic case. These estimates can be obtained by using similar Taylor expansion ideas than in the case $I_1$, see \cite{luirop}.

In the last case we need to show that
\[
(\alpha(x,t)-\alpha(z,t))M>I_3-(\alpha(x,t)-\alpha(z,t))(F(x,z,t)+\eta),
\]
where
\begin{align*}
I_3=&\frac{\alpha(x,t)-\alpha(z,t)}{2}\\
& \left(\sup_{B_\eps(x)}u_\eps(y,t-\frac{\eps^2}{2})+\inf_{B_\eps(x)}u_\eps(y,t-\frac{\eps^2}{2})-2\vint_{B_\eps(z)}u_\eps(y,t-\frac{\eps^2}{2})dy\right).
\end{align*} 
Again the extra error term compared to the elliptic case is on the scale of $r^{\delta-2}\eps^2$. By choosing a sequence $(x_j)$ such that $u_\eps(x_j)\rightarrow \sup_{B_\eps(x)}u_\eps$, we have 
\begin{align*}
\sup_{B_\eps(x)}&u_\eps-\vint_{B_\eps(z)}u_\eps(y,t-\frac{\eps^2}{2})dy\\
&=\vint_{B_\eps(z)}\lim_j(u_\eps(x_j)-u_\eps(y)-F(x_j,y,t-\frac{\eps^2}{2})+F(x_j,y,t-\frac{\eps^2}{2}))dy\\
&\leq M+\sup_{a\in B_\eps(x)}\vint_{B_\eps(z)}F(a,y,t-\frac{\eps^2}{2})dy.
\end{align*}
We also get
\[
\inf_{B_\eps(x)}u_\eps-\vint_{B_\eps(z)}u_\eps(y,t-\frac{\eps^2}{2})dy\leq M+\vint_{B_\eps(z)}\inf_{b\in B_\eps(x)}F(b,y,t-\frac{\eps^2}{2})dy,
\]
and finally
\begin{align*}
I_3\leq &\left(\frac{\alpha(x,t)-\alpha(z,t)}{2}\right)\\
& \left(2M+\sup_{a\in B_\eps(x)}\vint_{B_\eps(z)}F(a,y,t-\frac{\eps^2}{2})+\inf_{b\in B_\eps(x)}F(b,y,t-\frac{\eps^2}{2})dy\right).
\end{align*}
Hence, it is sufficient to show that
\[
f(x,z)>\frac12 \sup_{a\in B_\eps (x)}\left[\vint_{B_\eps(z)}f(a,y)+\inf_{b\in B_\eps (x)}f(b,y)dy\right]+2r^{\delta-2}\eps^2.
\]
The arguments are analogous to those used before, and we refer to \cite{luirop} for details.
\end{proof}

Next we consider the time direction. For the similar oscillation estimate in the PDE context, we refer to \cite[Lemma 4.3]{jins15} and \cite{barlesbl02}.

\begin{theorem}\label{time-Holder}
Let $B_{2r}(0)\times [-2r^2,0]\subset \Omega\times (-T,T)$ and $-r^2<t_0<t_1<0$. Then $u_\eps$ satisfies
\[
|u_\eps (x,t_1)-u_\epsilon (x,t_0)|\leq C(n)\frac{|t_1-t_0|^{\delta/2}}{r^\delta}+C'(n)\frac{\eps^\delta}{r^\delta}, 
\]
when $x\in B_r(0)$.
\end{theorem}

\begin{proof}
We define
\[
Q_r:=B_r(0)\times (-r^2,0).
\]
We want to show that the oscillation of $u$ in $Q_r$ is comparable with the oscillation of $u$ on the bottom of $Q_r$ by a constant depending only on the dimension $n$. The idea is to control the oscillation of $u_\eps$ by suitable comparison functions $\ol v$ and $\underline v$. We use the DPP together with suitable iteration to get estimates for $u_\eps$ and the comparison functions.

Denote
\[
A:=\osc_{B_r(0)\times \{-r^2\}}u_\eps
\]
and set the first comparison function $\ol v$ as
\[
\ol v(x,t)=\ol c+7r^{-2}At+2r^{-2}A|x|^2,
\]
where  $\ol c$ is chosen so that $\ol v(x,-r^2)\geq u_\eps(x,-r^2)$ for all $x\in B_r(0)$, and there is an equality for some $\ol x\in \overline B_r(0)$. Then actually $\ol x\in B_r(0)$, for otherwise 
\[
2A=\overline v(\ol x,-r^2)-\ol v(0,-r^2)\leq u(\ol x,-r^2)-u(0,-r^2)\leq A,
\]
a contradiction. First we estimate 
\begin{align*}
\beta(x,t)r^{-2}A\vint_{B_\eps (0)}|x+h|^2dh & \leq \beta(x,t)r^{-2}A\vint_{B_\eps (0)}|x|^2+2x\cdot h+|h|^2 dh\\
& \leq \beta(x,t)r^{-2}A(|x|^2+ \eps^2),
\end{align*}
Supposing that $|x|\geq \eps$ and using the previous estimate together with a simple calculation 
\[
\sup_{B_\eps(x)}|y|^2+\inf_{B_\eps(x)}|y|^2=|x+\eps|^2+|x-\eps|^2=2(|x|^2+\eps^2),
\]
we obtain
\begin{align*}
\frac{\alpha(x,t)}{2}&(\sup_{y\in B_\eps(x)}\ol v(y,t-\frac{\eps^2}{2})+\inf_{y\in B_\eps(x)}\ol v(y,t-\frac{\eps^2}{2}))+\beta(x,t)\vint_{B_\eps (x)}\ol v(y,t-\frac{\eps^2}{2})dy \\
& =2r^{-2}A\alpha(x,t)(|x|^2+\eps^2)+2r^{-2}A\beta(x,t)(|x|^2+c\eps^2)+7r^{-2}A(t-\frac{\eps^2}{2})+\ol c \\
& =\ol c+2r^{-2}A|x|^2+7r^{-2}At+\left(2r^{-2}A\alpha(x,t)+2r^{-2}A\beta(x,t)-\frac{7r^{-2}A}{2}\right)\eps^2\\
& <\ol v(x,t).
\end{align*}
One can easily see that the same inequality holds when $|x|<\eps$.

We want to show that 
\[
M:=\sup_{Q_r} (u_\eps-\ol v)\leq 0.
\]
Suppose not, so that $M>0$. By using the DPP for $u_\eps$ we get 
\begin{align*}
& u_\eps(x,t)-\ol v(x,t)\\
& \leq Tu_\eps(x,t)-T\ol v(x,t)\\
& \leq \alpha(x,t)\sup_{B_\eps (x)}(u_\eps(y,t-\frac{\eps^2}{2})-\ol v(y,t-\frac{\eps^2}{2}))\\
& +\beta(x,t)\vint_{B_\eps (x)}(u_\eps(y,t-\frac{\eps^2}{2})-\ol v(y,t-\frac{\eps^2}{2}))dy\\
& \leq \alpha(x,t)M+\beta(x,t)\vint_{B_\eps (x)}(u_\eps(y,t-\frac{\eps^2}{2})-\ol v(y,t-\frac{\eps^2}{2}))dy.
\end{align*}
Since we can find a sequence $(x_j,t_j)\subset \Omega\times (-T,T)$ such that $(x_j,t_j)\rightarrow (x_0,t_0)$ and $(u_\eps-\ol v)(x_j,t_j)\rightarrow M$, by absolute continuity of the integral we have   
\[
\vint_{B_\eps(x_0)}(u_\eps-\ol v)(y,t_0)dy=\lim_j \vint_{B_\eps(x_j)}(u_\eps-\ol v)(y,t_j)dy=M.
\]
Hence the set 
\[
G:=\left\{(x,t): u_\eps(x,t)-\ol v(x,t)=M\right\} 
\]
is non-empty, and if $(x_0,t_0)\in G$, then $(u_\eps-\ol w)(y,t_0)=M$ for almost all $y\in B_\eps(x_0)$. This contradicts the assumption that $G$ is bounded. Hence $M\leq 0$.

Similarly, we can show that for 
\[
\underline v(x,t)=\underline c-7r^{-2}At-2r^{-2}A|x|^2
\]
we have $\underline w\leq u$ in the cylinder $Q_r$. Hence
\[
\ol v(\ol x,-r^2)-\underline v(\underline x,-r^2)\leq \osc_{B_r(0)\times \{-r^2\}} u_\eps,
\] 
so
\[
\ol c-\underline c\leq 11A.
\]
Finally, we get
\[
\osc_{Q_r}u\leq \sup \ol v-\inf \underline v\leq \ol c-\underline c+4A\leq CA,
\]
so the oscillation in the cylinder $Q_r$ is comparable with the oscillation on the bottom of the cylinder.
\end{proof}

\begin{remark}
Another way to prove the previous lemma is to use the same technique that was used in the proof of Theorem \ref{aika-Lip}.
\end{remark}
Combining the two previous theorems, we get local H\"older continuity for the $p(x,t)$-game.

\begin{theorem}\label{Holder}
Under the conditions of Theorem \ref{time-Holder}, $u_\eps$ satisfies the H\"older estimate
\[
|u_\eps (x,t_1)-u_\epsilon (y,t_0)|\leq C(n)\frac{|x-y|^\delta+|t_1-t_0|^{\delta/2}}{r^\delta}+C'(n)\frac{\eps^\delta}{r^\delta}. 
\]
\end{theorem}

\subsection{Harnack's inequality} In this subsection we assume that $u_\eps>0$. We are going to prove Harnack's inequality for $u_\eps$, Theorem \ref{Harnack}, by using a well known iteration technique. Besides H\"older continuity, we need two lemmas to control the iteration process. We assume for function $p:\Omega_T\rightarrow (2,\infty)$ that 
\[
\inf p>2,
\]
which implies that $\inf \alpha>0$. This requirement is not absolutely necessary, but makes the proof less technical.

Since H\"older continuity for $u_\eps$ breaks down at the $\eps$-scale, we need a rough estimate to control the oscillation of the value function at this scale.

\begin{lemma}\label{lokaali}
If $\frac a2 \eps^2>t_2-t_1>0$ for $a\in \Z_+$, and $\left|x-y\right|<2(t_2-t_1)/\eps$, then
\[
u_\eps(x,t_2)\geq \left(\frac{\inf \alpha}{2}\right)^{a}u_\eps(y,t_1).
\]
\end{lemma}

\begin{proof}
When the game starts from $(x,t_2)$, Player I uses a strategy in which she takes $\frac{\left|x-y\right|}{a}$-steps towards $y$ and steps to $y$ if possible. We stop the game when the token hits the time level $t_1$, and denote the stopping time by $\tau^*$. By simply estimating  the probability that the first $a$ moves are tug-of-war won by Player I and using Lemma \ref{fix.strategia}, we obtain
\begin{align*}
u_\eps(x,t_2)&\geq \inf_{S_{\text{II}}}\mathbb{E}^{(x,t_2)}_{S^{\text{0}}_\text{I},S_{\text{II}}}[F(x_{\tau_{t^*}},t_2-\frac{\tau^*}{2}\eps^2)]\\
&\geq \left(\frac{\inf \alpha}{2}\right)^{a}u_\eps(y,t_1).\qedhere
\end{align*}
\end{proof}

Another lemma needed for Theorem \ref{Harnack} gives estimates for the infimum of $u_\eps$. We use a comparison function which is often used in the literature to get Harnack estimates for parabolic equations. 

\begin{lemma}\label{pos.laajennus}
When $x_0\in B_{2R}(z)\subset \Omega$ for $R\leq 1$, $r\in [9\eps,R)$ and $t_0\geq 0$, then
\[
\inf_{y\in B_r(z)}u_\eps(y,t_0)\leq C(n)r^{-2(n+1)^2}u_\eps(x_0,t_0+R^2).
\]
\end{lemma}

\begin{proof}
Without a loss of generality, we may assume that $z=0$ and $t_0=0$. Consider a comparison function
\[
\Psi(x,t)=\left(\frac19\right)^3\inf_{y\in B_r(0)}u_\eps(y,0)\frac{(\frac13 r)^{2(n+1)^2}}{(t+(\frac13 r)^2)^{(n+1)^2}}\left(9-\frac{|x|^2}{t+(\frac13 r)^2}\right)^2_+
\]
in $\Omega_T$. We have
\[
\max_{x\in B_r(0)}\Psi (x,0)=\frac19\inf_{y\in B_r(0)}u_\eps(y,0), 
\]
and $\Psi(x,0)=0$ when $|x-z|\geq r$.

When $x\in B_{2R}(0)$ and $R^2\leq t\leq 2R^2$, we get
\begin{align*}
\Psi(x,t)&\geq \left(\frac19\right)^3\inf_{y\in B_r(0)}u_\eps(y,0)\frac{(\frac13 r)^{2(n+1)^2}}{(2R^2+(\frac13 R)^2)^{(n+1)^2}}\left(9-\frac{4R^2}{R^2}\right)^2_+\\
& \geq \left(\frac19\right)^3 3^{-3(n+1)^2}r^{2(n+1)^2}\inf_{y\in B_r(0)}u_\eps(y,0).
\end{align*}

We use a martingale argument to show that 
\[
u_\eps(x,t)>\Psi(x,t),
\]
when $x\in \Omega$ and $t>0$. Let us start the game from $(x_0,\widetilde{t})$, where $\widetilde{t}=R^2$. The fixed strategy $S^0_\text{\emph{I}}$ of Player I is to push towards $0\in \Omega$ and stay there if possible. We show in Appendix that the function $\Psi$ satisfies the following inequalities:

Case 1) If $x=0$ and $t\geq \eps^2 /2$, then for $e\in \R^n$, $|e|=1$,
\[
\frac12[\Psi(0,t-\frac{\eps^2}{2})+\Psi(\eps e,t-\frac{\eps^2}{2})]\geq \Psi(0,t).
\]

Case 2) If $0<|x|<\eps$, then
\[
\frac12[\Psi(0,t-\frac{\eps^2}{2})+\Psi(x+\frac{x}{|x|}\eps,t-\frac{\eps^2}{2})]\geq \Psi(x,t).
\]

Case 3) If $|x|\geq \eps$, then
\[
\frac12[\Psi(x+\frac{x}{|x|}\eps,t-\frac{\eps^2}{2})+\Psi(x-\frac{x}{|x|}\eps,t-\frac{\eps^2}{2})]\geq \Psi(x,t).
\]
The previous three inequalities guarantee that $\Psi$ satisfies
\[
\Psi(x,t)\leq \frac12 (\sup_{y\in B_\eps(x)}\Psi(y,t-\frac{\eps^2}{2})+\inf_{y\in B_\eps(x)}\Psi(y,t-\frac{\eps^2}{2})).
\]
In the appendix we also show that $\Psi$  is a subsolution to the scaled heat equation
\[
(n+2)u_t(x,t)=\Delta u(x,t).
\]
According to \cite{manfredipr10}, this implies 
\[
\Psi(x,t)\leq \vint_{B_\eps(x)}\Psi\left(y,t-\frac{\eps^2}{2}\right)dy+o(\eps^2), 
\]
when $x\in \Omega$ and $t>0$. Denote $t_k:=\widetilde{t}-k(\eps^2/2)$. For arbitrary $\eta>0$, we obtain
\begin{align*}
\mathbb{E}_{S^0_\text{\emph{I}},S_\text{\emph{II}}}&[\Psi(x_{k+1},t_{k+1})|(x_0,\widetilde{t}),...,(x_k,t_k)] \\
& \geq \alpha(x)\Psi(x_k,t_k)+\beta(x)\vint_{B_\eps(x_k)}\Psi\left(y,t_{k+1}\right)dy-\frac{\eta}{2R^2} \eps^2 k\\
& \geq \Psi(x_k,t_k)-\frac{\eta}{2R^2} \eps^2,
\end{align*} 
when $\eps$ is sufficiently small. According to Lemma \ref{fix.strategia}, $u_\eps$ satisfies  
\[
\mathbb{E}_{S^0_\text{\emph{I}},S_\text{\emph{II}}}[u_\eps(x_{k+1},t_{k+1})|(x_0,\widetilde{t}),...,(x_k,t_k)]\leq u_\eps(x_k,t_k).
\]
Hence $M_k:=u_\eps(x_k,t_k)-\Psi(x_k,t_k)-\frac{\eta}{2R^2} \eps^2 k$ is a supermartingale. Let us stop the game when either $\Psi=0$ or $t_k=0$. Denote the stopping time by $\tau^*$. We have
\[
-\eta\leq\mathbb{E}_{S^0_\text{\emph{I}},S_\text{\emph{II}}}[M_{\tau^*}|(x_0,\widetilde{t}),...,(x_{\tau^* -1},\widetilde{t}-\frac{\tau^*-1}{2}\eps^2]\leq M_0=u_\eps(x_0,\widetilde{t})-\Psi(x_0,\widetilde{t}).
\]
Since $\eta>0$ was arbitrary, we obtain
\[
u_\eps(x_0,\widetilde{t})-\Psi(x_0,\widetilde{t})\geq 0.
\]
Hence
\[
\inf_{y\in B_r(z)}u_\eps(y,t_0)\leq C(n)r^{-2(n+1)^2}u_\eps(x_0,t_0+R^2).\qedhere
\]
\end{proof}

Using H\"older estimate together with Lemmas \ref{lokaali} and \ref{pos.laajennus}, we get Harnack's inequality for $u_\eps$.

\begin{theorem}\label{Harnack}
If $B_{10r}(0)\times [t_0-r^2,t_0]\subset \Omega_T$, then for sufficiently small $\eps>0$, $u_\eps$ satisfies Harnack's inequality
\[
\sup_{x\in B_r(0)}u_\eps(x,t_0-r^2)\leq C(n)\inf_{x\in B_r(0)}u_\eps(x,t_0).
\]
\end{theorem}

\begin{proof}
By scaling, we may assume that there is a point $x_1\in B_r(0)$ such that
\[
1=u_\eps(x_1,t_0)<2\inf_{x\in B_r(0)}u_\eps(x,t_0).
\]
Let $R_k:=2^{1-k}r$ for all natural numbers $k\geq 2$, and pick $x_2,x_3,...\in \Omega$ such that 
\[
M_1:=u_\eps(x_2,t_0)=\sup_{x\in B_r(x_1)}u_\eps(x,t_0),
\]
and for $k\geq 2$
\[
M_k:=u_\eps(x_{k+1},t_0-r^2+R^2_{2k-1})=\sup_{x\in B_{R_k}(x_k)}(x,t_0-r^2+R^2_{2k-1}).
\]
Let $\eta=(2^{1+3(n+1)^2}C)^{-1}$, where $C=C(n)$ is a constant from the H\"older and infimum estimates. We are going to show that 
\begin{equation}\label{ylaraja}
M_1< \eta^{-1-3(n+1)^2 \delta^{-1}},
\end{equation}
where $\delta$ is a H\"{o}lder exponent for $u_\eps$.

On the contrary, suppose that inequality \eqref{ylaraja} does not hold. Let us show by induction that the counter assumption yields
\begin{equation}\label{M}
M_k\geq (2C\eta)^{-k+1} \eta^{-1-3(n+1)^2 \delta^{-1}}=2C(\eta^{1/\delta}R_{k+1})^{-3(n+1)^2}.
\end{equation}
The case $k=1$ is clear, so assume that the inequality holds for $M_{k-1}$. Then
\begin{align}\label{inf puoli}
\inf_{B_{\eta^{1/\delta} R_k}(x_k)}u_\eps(x,t_0-r^2+R^2_{2(k-1)})&\leq\frac{M_{k-1}}{2}\nonumber\\
& = \frac{u_\eps (x_k,t_0-r^2+R^2_{2(k-1)-1})}{2},
\end{align}
where we first used Lemma \ref{pos.laajennus} and then the induction assumption.

H\"{o}lder estimate gives
\begin{align*}
\text{osc}&(u_\eps,B_{\eta^{1/\delta} R_k}(x_k)\times \{t_0-r^2+R^2_{2(k-1)}\})\\
& \leq C\eta\ \text{osc}(u_\eps,B_{R_k}(x_k)\times \{t_0-r^2+R^2_{2k-1}\}),
\end{align*}
so we get
\begin{align*}
\text{osc}&(u_\eps,B_{R_k}(x_k)\times \{t_0-r^2+R^2_{2k-1}\})\\ 
&\geq (C\eta)^{-1}\text{osc}(u_\eps,B_{\eta^{1/\delta} R_k}(x_k)\times \{t_0-r^2+R^2_{2(k-1)}\})\\
& \geq (2C\eta)^{-1}M_{k-1}\\
& \geq (2C\eta)^{-k+1}M_1,
\end{align*}
and the induction is complete.

Take $k_0$ such that $\eta^{1/\delta} R_{k_0}\in (10\eps,20\eps]$. Then
\[
R^2_{2(k_0-1)}\leq 100\eta^{-2/\delta}\eps^2\leq (2^{8+3(n+1)^2}C)^{2/\delta}\eps^2,
\]
and we obtain
\begin{align*}
\left(\frac{\inf \alpha}{2}\right)^{-2(2^{8+3(n+1)^2}C)^{2/\delta}}&\geq \frac{\sup_{B_{R_{k_0-1}}(x_{k_0-1})}u_\eps(x,t_0-r^2+R^2_{k_0-1})}{\inf_{B_{\eta^{1/\delta} R_{k_0}}(x_{k_0})}u_\eps(x,t_0-r^2+R^2_{k_0})}\\
&\geq \frac{u_\eps (x_{k_0-1},t_0-r^2+R^2_{k_0-1})}{C(\eta^{1/\delta} R_{k_0})^{-2(n+1)^2}}\\
&=\frac{M_{k_0-2}}{C(\eta^{1/\delta} R_{k_0})^{-2(n+1)^2}}\\
& \geq \frac{(2C\eta)^{3-k_0}M_1}{C(\eta^{1/\delta} 2^{1-k_0})^{-2(n+1)^2}}\\
& =\widehat{C}(n)2^{(n+1)^2 k_0},
\end{align*}
which is a contradiction when $k_0$ is big enough, or in other words, when $\eps$ is small enough. Therefore inequality \eqref{ylaraja} holds and the proof is complete.
\end{proof}

\section{Uniform convergence to viscosity solution}\label{uniform}
In Section \ref{uniq} we will show that if the function $p$ is Lipschitz continuous, there is a unique viscosity solution $u$ to the boundary value problem
\[
\begin{split}
\begin{cases}
(n+p(x,t))u_t
=\Delta^N_{p(x,t)}u  ,\quad  &\textrm{for}\quad (x,t)\in \Omega_T,\\
u   =  F
,\quad  &\textrm{for}\quad  (x,t)\in \partial_p \Omega_T,
\end{cases}
\end{split}
\]
where $F$ is continuous and bounded. Let $(u_{\eps_j})$, $\eps_j\rightarrow 0$, be a sequence of value functions of the $p(x,t)$-game with final payoff equal to $F$ on the parabolic boundary strip $\Gamma^\eps_T$. In this section we show that $u_{\eps_j}\rightarrow u$ uniformly on $\overline{\Omega}_T$. The most notable difference is that now we don't have translation invariance at our disposal. Instead, we will make use of local Hölder continuity of functions $u_{\eps_j}$, see Theorem \ref{Holder}. We assume during the rest of the paper that $\Omega$ satisfies exterior sphere condition.

First we need the following Arzel\'a-Ascoli-type lemma. For the proof in the elliptic context, see \cite[Lemma 4.2]{manfredipr12}.

\begin{lemma}\label{Ascoli}
Let $\left\{u_\eps:\overline{\Omega}_T\rightarrow \R,\ \eps>0\right\}$ be a uniformly bounded set of functions such that given $\eta>0$, there are constants $r_0$ and $\eps_0$ such that for every $\eps<\eps_0$ and any $(x,t),(y,s)\in \overline{\Omega}_T$ with 
\[
\left|x-y\right|+|t-s|<r_0
\]
it holds that
\[
\left|u_\eps(x,t)-u_\eps(y,s)\right|<\eta. 
\]
Then there exists a uniformly continuous function $v:\overline{\Omega}_T\rightarrow \R$ and a subsequence still denoted by $(u_\eps)$ such that $u_\eps\rightarrow v$ uniformly in $\overline{\Omega}_T$ as $\eps\rightarrow 0$.
\end{lemma}

The plan is to first show that the sequence $(u_{\eps_j})$ satisfies the conditions of Lemma \ref{Ascoli}, and then show that the uniform limit $v$ is a viscosity solution to
\[
(n+p(x,t))v_t=\Delta^N_{p(x,t)}v
\]
with boundary data $F$. By using the uniqueness result of Section \ref{uniq}, we will conclude that $v=u$ on $\overline{\Omega}_T$. Our proofs yield that an arbitrary subsequence of $(u_{\eps_j})$ has a uniformly convergent subsequence. Hence, by uniqueness of $u$, the sequence $(u_{\eps_j})$ itself converges uniformly to $u$. 

To show that the sequence $(u_{\eps_j})$ satisfies the conditions of Lemma \ref{Ascoli}, we first need the following technical lemma, in which the function $p(x,t)$ does not cause extra difficulties compared to the case where $p>2$ is a constant. The method  for proof has been used before for different games, see \cite[Lemma 4.9]{manfredipr10} and \cite[Lemma 4.5]{manfredipr12}.

\begin{lemma}
For arbitrary $\eta>0$, there are $r_0>0$ and $\eps_1>0$ such that when $(y,t)\in \partial_p \Omega_T$, $(x,s)\in \Omega_T$, $\eps<\eps_1$ and $|y-x|+|t-s|<r_0$, we have 
\[
|u_\eps(y,t)-u_\eps(x,s)|<\eta.
\]
\end{lemma}

\begin{proof}
If $(y,t)$ is on the bottom of the cylinder $\Omega_T$, the result follows from Theorem \ref{time-Holder}. Assume next that $y\in \partial \Omega$. It is enough to verify the case $t=s=:t_0$, since otherwise triangle inequality gives
\[
|u_{\eps_j}(x,t)-u_{\eps_j}(y,s)|\leq |u_{\eps_j}(x,t)-u_{\eps_j}(y,t)|+|u_{\eps_j}(y,t)-u_{\eps_j}(y,s)|,
\]
and the last term can be estimated by using uniform continuity of the boundary data. 

Since $\Omega$ satisfies the exterior sphere condition, we have $y\in \partial B_\delta(z)$ for some $B_\delta(z)\subset \R^n\setminus \Omega$. Let us start the game from $(x,t)=:(x_0,t_0)$ and fix for Player I a strategy $S^0_\text{I}$ of pulling towards $z$. Player II uses a strategy $S_\text{II}$. Then,
\begin{align*}
\mathbb{E}& ^{(x_0,t_0)}_{S^0_\text{I},S_\text{II}}[|x_k-z||x_0,...,x_{k-1}]\\
& \leq \frac{\alpha(x_{k-1},t_{k-1})}{2}\left(|x_{k-1}-z|+\eps+|x_{k-1}-z|-\eps\right)\\
& +\beta(x_{k-1},t_{k-1})\vint_{B_\eps(x_{k-1})}|x-z|dx\\
&\leq |x_{k-1}-z|+C\eps^2,
\end{align*}
where $C$ does not depend on $\eps$. Therefore, $M_k=|x_k-z|-C\eps^2k$ is a supermartingale. Jensen's inequality gives
\[
\mathbb{E}^{(x_0,t_0)}_{S^0_\text{I},S_\text{II}}[|x_\tau-z|+|t_\tau-t_0|^{\frac12}]\leq |x_0-z|+C\eps\left(\mathbb{E}^{(x_0,t_0)}_{S^0_\text{I},S_\text{II}}[\tau]\right)^{\frac12}.
\]
Suppose that for the stopping time $\tau$ we have the estimate
\begin{equation}\label{stopping}
\mathbb{E}^{(x_0,t_0)}_{S^0_\text{I},S_\text{II}}[\tau]\leq \frac{C(R/\delta)\dist(\partial B_\delta (z),x_0)+o(1)}{\eps^2},
\end{equation}
where $R>0$ is chosen so that $\Omega\subset B_R(z)$, and $o(1)\rightarrow 0$ when $\eps\rightarrow 0$. Then we have
\[
\mathbb{E}^{(x_0,t_0)}_{S^0_\text{I},S_\text{II}}[|x_\tau-z|+|t_\tau-t_0|^{\frac12}]\leq |x_0-z|+C(R/\delta)|x_0-y|+o(1),
\]
and the proof is complete by uniform continuity of the boundary function $F$.

It remains to justify estimate \eqref{stopping}. In $\Omega$, let $v$ be a solution to the problem
\begin{displaymath}
\left\{ \begin{array}{ll}
\Delta v=-2(n+2) & \textrm{in}\ B_{R+\eps}\setminus \overline{B}_{r}(z),\\
v=0 & \text{on}\ \partial B_{r}(z),\\
\frac{\partial v}{\partial \nu}=0 & \text{on}\ \partial B_{R+\eps}(z), 
\end{array} \right.
\end{displaymath}
where $\frac{\partial v}{\partial \nu}$ is the normal derivative. The function $v$ satisfies
\begin{equation}\label{poisson}
v(x)=\vint_{B_\epsilon(x)}v\ dy+\eps^2,
\end{equation}
and it can be extended as a solution to the same equation in $\overline B_{r(z)}\setminus \overline B_{{r-\eps}(z)}$ so that equation \eqref{poisson} holds also near the boundary $\partial B_r(z)$. 

By concavity of $v$, it follows from \eqref{poisson} that $(v(x_k)+k\eps^2)$ is a supermartingale. Define a new stopping time $\tau^*$,
\[
\tau^*=\inf\{k\ :\ x_k\in \overline{B}_\delta (z)\}.
\] 
Since 
\[
v(x_0)\leq C(R/\delta)\dist(\partial B_\delta (z),x_0),
\]
we have 
\[
\mathbb{E}^{x_0}[\tau^*]\leq \frac{v(x_0)-\mathbb{E}[v(x_{\tau^*})]}{\eps^2}\leq \frac{C(R/\delta)\dist(\partial B_\delta (z),x_0)+o(1)}{\eps^2}.
\]
Since the function $v$ is concave in $r=|x-z|$ and $\tau\leq \tau^*$, we obtain estimate \eqref{stopping}, and the proof is complete. 
\end{proof}

\begin{lemma}
The sequence $(u_\eps)$ of value functions satisfies the conditions of Lemma \ref{Ascoli}.
\end{lemma}

\begin{proof} 
Since $u_\eps\leq \max F$, the sequence $(u_\eps)$ is uniformly bounded. For asymptotic uniform continuity, fix $\eta$. Since $u$ is uniformly continuous in $\Omega_T\times \Gamma_\eps$, there is $r_1>0$ such that $(x,t), (y,s)\in \Omega_T\times \Gamma_\eps$, 
\[
\left|x-y\right|+|t-s|<r_1,
\]   
implies
\[
\left|u(x,t)-u(y,s)\right|<\eta/2.
\]
When $x,y\in \partial B_R(0)$, the same estimate holds between $u_\eps(x)$ and $u_\epsilon(y)$ for all $0<\eps<R$, since $u_\eps=u$ on $\Gamma_\eps$.

When $y\in \Gamma_\eps$ and $x\in \Omega_T$, by the previous lemma there are $r_0>0$ and $\eps_1>0$ such that when $|y-x|+|t-s|<r_0$, we have 
\[
|u_\eps(y,t)-u_\eps(x,s)|<\eta/2.
\]
If $x,y\in \Omega_T$ and $\dist(\{x,y\},\Gamma_\eps)<r_0/2$, then by using the triangle inequality with a boundary point, we obtain $|u_\eps(y)-u_\eps(x)|<\eta$. 

Finally, assume that $\dist(\{x,y\},\Gamma_\eps)\geq r_0/2$. By local H\"older continuity there is $\eps_2>0$ such that when $\eps<\eps_1$, we have 
\[
|u_\eps(y,t)-u_\eps(x,s)|<\eta. 
\]
The proof is complete by taking $\eps_0=\min(\eps_1,\eps_2)$.
\end{proof}
We have shown that the sequence $(u_\eps)$ converges uniformly towards a uniformly continuous limit function $v$, and next we show that the function is a viscosity solution to the normalized parabolic $p(x,t)$-equation.

Below we denote by
$\lambda_{\textrm{max}}((p(x,t)-2)D^2\phi(x,t))$, and
$\lambda_{\textrm{min}}((p(x,t)-2)D^2\phi(x,t))$ the largest, and the
smallest of the eigenvalues to the symmetric matrix
$(p(x,t)-2)D^2\phi(x,t)\in \R^{n\times n}$ for a smooth test function.
\begin{definition}
\label{def:viscosity-solution}
A function $u : \Om_T\to \R$ is a viscosity solution to
\[
\begin{split}
(n+p(x,t))u_t=\Delta u+(p(x,t)-2) \Delta_{\infty}^N u,
\end{split}
\]
if $u$ is continuous and whenever $(x_0,
t_0)\in \Om_T$ and $\phi \in C^2(\Om_T)$ is such that
\begin{enumerate}
\item[i)]  $u(x_0, t_0) = \phi(x_0, t_0)$,
\item[ii)] $\phi(x, t) > u(x, t)$ for $(x, t) \in \Om_T,\ (x,t)\neq (x_0,t_0)$,
\end{enumerate}
then we have at the point $(x_0, t_0)$
\[
\begin{split}
\begin{cases}
(n+p(x,t))\phi_t
\leq (p(x,t)-2)\Delta^N_\infty \phi  +\Delta \phi  ,\quad  &\textrm{if}\quad \nabla \phi(x_0,t_0)\neq 0,\\
(n+p(x,t))\phi_t   \leq  \lambda_\textrm{max}((p(x,t)-2)D^2\phi )+\Delta \phi
,\quad  &\textrm{if}\quad  \nabla \phi(x_0,t_0)  =0.
\end{cases}
\end{split}
\]
Moreover, we require that when touching $u$ with a test function
from below all the inequalities are reversed and
$\lambda_\textrm{max}((p(x,t)-2)D^2\phi)$ is replaced by
$\lambda_{\textrm{min}}((p(x,t)-2)D^2\phi)$.
\end{definition}

\begin{lemma}\label{raja}
The limit function $v$ is a viscosity solution to 
\[
\begin{split}
(n+p(x,t))u_t=\Delta u+(p(x,t)-2) \Delta_{\infty}^N u,
\end{split}
\]
with boundary data $F$.
\end{lemma}

\begin{proof}
We only show that the function $v$ is a viscosity supersolution. (Showing that $v$ is a subsolution is similar.) Choose $(x,t)\in Q_R$ and $\vp\in C^2$ touching $v$ from below at $(x,t)$. We need to show that 
\begin{equation}\label{visc}
\frac{\beta(x,t)}{2(n+2)}\left((p(x,t)-2)\Delta^N_\infty \vp(x,t)+\Delta \vp(x,t)-(n+p(x,t))\vp_t(x,t)\right)\leq 0.
\end{equation}
As a direct consequence of \cite[Theorem 2.4]{manfredipr10}, we have
\begin{align*}
&\frac{\alpha(x,t)}{2}\left\{\sup_{B_\eps (x)}\vp(y,t-\frac{\eps^2}{2})+\inf_{B_\eps (x)}\vp(y,t-\frac{\eps^2}{2})\right\}\\
&\phantom{{}=\sup_{B_\eps (x)}}+\beta(x,t) \vint_{B_\eps (x)}\vp(y,t-\frac{\eps^2}{2})dy-\vp(x,t)\notag\\
& \geq \frac{\beta(x,t) \eps^2}{2(n+2)}\bigg((p(x,t)-2)\left\langle D^2 \vp(x,t)\left(\frac{\overline{x}^\eps-x}{\left|\overline{x}^\eps-x\right|}\right),\left(\frac{\overline{x}^\eps-x}{\left|\overline{x}^\eps-x\right|}\right)\right\rangle\\
&\phantom{{}=\frac{\beta(x,t) \eps^2}{2(n+2)}}+\Delta\vp(x,t)-(n+p(x,t))\vp_t(x,t)\bigg)+o(\eps^2),
\end{align*}
where $\overline{x}^\eps\in B_\eps(x)$ is nearly to the direction of $\nabla \vp(x)$. 

By the uniform convergence, there is a sequence $(x_\eps,t_\eps)\rightarrow (x,t)$ such that when $(y,s)$ is near $(x_\eps,t_\eps)$, we have
\[
u_\eps(y,s)-\vp(y,s)\geq u_\eps(x_\eps,t_\eps)-\vp(x_\eps,t_\eps)-\eta_\eps.
\]
Setting $\widetilde{\vp}=\vp+u_\eps(x_\eps,t_\eps)-\vp(x_\eps,t_\eps)$ we have
\[
u_\eps(x_\eps,t_\eps)=\widetilde{\vp}(x_\eps,t_\eps),\ u_\eps(y,s)\geq \widetilde{\vp}(y,s)-\eta_\eps.
\]
We get
\begin{align*}
\eta_\eps &\geq \frac{\alpha(x,t_\eps)}{2}\left\{\sup_{B_\eps (x)}\widetilde{\vp}(y,t_\eps-\frac{\eps^2}{2})+\inf_{B_\eps (x)}\widetilde{\vp}(y,t_\eps-\frac{\eps^2}{2})\right\}\\
&\phantom{{}=\sup_{B_\eps (x)}}+\beta(x,t_\eps) \vint_{B_\eps (x)}\widetilde{\vp}(y,t_\eps-\frac{\eps^2}{2})dy-\widetilde{\vp}(x_\eps,t_\eps)\notag
\end{align*}
Let us first assume that $\nabla \vp(x,t)\neq 0$. Then, since we can choose $\eta_\eps=o(\eps^2)$, we obtain
\begin{align*}
0 &\geq \frac{\beta(x_\eps,t_\eps) \eps^2}{2(n+2)}\bigg((p(x_\eps,t_\eps)-2)\left\langle D^2 \vp(x_\eps,t_\eps)\left(\frac{\overline{x}^\eps-x_\eps}{\left|\overline{x}^\eps-x_\eps\right|}\right),\left(\frac{\overline{x}^\eps-x_\eps}{\left|\overline{x}^\eps-x_\eps\right|}\right)\right\rangle\\
&\phantom{{}=\frac{\beta(x) \eps^2}{2(n+2)}}+\Delta\vp(x_\eps,t_\eps)-(n+p(x_\eps,t_\eps))\vp_t(x_\eps,t_\eps)\bigg)+o(\eps^2).
\end{align*}
When $\eps\rightarrow 0$, it follows that
\[
\frac{\beta(x,t)}{2(n+2)}\left((p(x,t)-2)\Delta^N_\infty \vp(x)+\Delta \vp(x)-(n+p(x,t))\vp_t(x,t)\right)\leq 0.
\]

When $\nabla \vp(x,t)=0$, also $D^2\vp(x,t)=0$ (see Lemma \ref{lemma:reduced-test-functions} below), and it is easy to verify the required inequality $\vp_t(x,t)\geq 0$.
\end{proof}

\section{Uniqueness for $p(x,t)$-equation}\label{uniq}

In this section we assume that the function $p$ is Lipschitz continuous in $\Omega_T$ with Lipschitz constant $C_1$. We prove that there is a unique viscosity solution to
\begin{equation}
\label{eq:equation}
\begin{split}
(n+p(x,t))u_t=\Delta_{p(x,t)}^N u
\end{split}
\end{equation}
with classical Dirichlet boundary conditions. Existence is well known, and in fact the previous section provided a game-theoretic proof. 

The technique for uniqueness is well known; $p(x,t)$ causes slight modifications. For the convenience of the reader, we give the details. For additional literature, see \cite{juutinen14, imberts13, banerjeeg13, banerjeeg15a}. 

The parabolic equation \eqref{eq:equation} is discontinuous when the gradient vanishes. We recall the definition of viscosity solution
based on semicontinuous extensions of the operator, and refer the
reader to  Chen-Giga-Goto~\cite{chengg91},
Evans-Spruck~\cite{evanss91}, and Giga's monograph \cite{giga06}.

The next lemma allows us reduce the test functions in the case $\nabla \phi(x_0,t_0)=0$ and only test by those having $D^2\phi(x_0,t_0)=0$.

\begin{lemma}
\label{lemma:reduced-test-functions}
A function $u : \Om_T\to \R$ is a viscosity solution to \eqref{eq:equation}
if $u$ is continuous and
whenever $(x_0, t_0)\in \Om_T$ and $\phi \in C^2(\Om_T)$ is such that
\begin{enumerate}
\item[i)]  $u(x_0, t_0) = \phi(x_0, t_0)$,
\item[ii)] $\phi(x, t) > u(x, t)$ for $(x, t) \in \Om_T,\ (x,t)\neq (x_0,t_0)$,
\end{enumerate}
then at the point $(x_0, t_0)$ we have
$$ \left\{\!\!\!
\begin{array}{rclcl}
(n+p(x,t))\phi_t \!\!
&\leq& \!\!(p(x,t)-2)\Delta^N_\infty \phi  +\Delta \phi,\!\!\!\!&\textrm{if}& \nabla
\phi(x_0,t_0)\neq 0,\notag \\[2pt]
\phi_t(x_0,t_0) \!\! &\leq &\!\! 0,\!\!\!\!&\textrm{if}& \nabla \phi (x_0,t_0)=0,\, \textrm{and} \,
D^2\phi(x_0,t_0)=0.\notag
\end{array} \right.
 $$
We also require that when testing from below all the inequalities are reversed.
\end{lemma}
\begin{proof}
The proof is by contradiction: We assume that
$u$ satisfies the conditions in the statement but still fails to be a
viscosity solution in the sense of
Definition~\ref{def:viscosity-solution}. If this is the case, we
must have  $\phi \in C^2(\Om_T)$, $(x_0,t_0)\in \Om_T$ and $\eta>0$ such
that
\begin{enumerate}
\item[i)]  $u(x_0, t_0) = \phi(x_0, t_0)$,
\item[ii)] $\phi(x, t) > u(x, t)$ for $(x, t) \in \Om_T,\ (x,t)\neq (x_0,t_0)$,
\end{enumerate}
for which $\nabla \phi(x_0,t_0)=0,\ D^2\phi(x_0,t_0)\neq 0$ and
\begin{align}
\label{eq:counter-proposition}
(n&+p(x_0,t_0))\phi_t(x_0,t_0)  -\eta\nonumber\\ 
&>  \lambda_\textrm{min}((p(x_0,t_0)-2)D^2\phi(x_0,t_0))+\Delta \phi(x_0,t_0),
\end{align}
or the analogous inequality when testing from below (in this case
the argument is symmetric and we omit it). Let
\begin{equation}\label{sakko}
w_j(x,t,y,s)=u(x,t)-\phi(y,s)-\Big(\frac{j^2}{4}\abs{x-y}^4+\frac{j}{2}\abs{t-s}^2\Big)
\end{equation}
and denote by $(x_j,t_j,y_j,s_j)$ the maximum point of $w_j$ in
$\ol \Om_T\times \ol \Om_T$. Since $(x_0,t_0)$ is a local maximum
for $u-\phi$, we may assume that
\[
\begin{split}
(x_j,t_j,y_j,s_j)\to (x_0,t_0,x_0,t_0)\quad \textrm{as} \quad j\to \infty,
\end{split}
\]
and $(x_j,t_j)\, ,(y_j,s_j)\in \Om_T$ for all large $j$, similarly to \cite{juutinenk06}. Since $(x_0,t_0)$ is a local maximum of $u-\phi$, it follows from \eqref{sakko} that 
\[
\frac{j^2}{4}\abs{x_j-y_j}^4\rightarrow 0\ \text{and}\ \frac{j}{2}\abs{t_j-s_j}^2\rightarrow 0,
\]
when $j\rightarrow \infty$. If not, there would be $\alpha>0$ and subsequences $(x_j),...(s_j)$ such that
\[
\frac{j^2}{4}\abs{x-y}^4+\frac{j}{2}\abs{t-s}^2>\alpha.
\] 
Let $U_\alpha$ be a neighborhood of $(x_0,t_0)$ where oscillation of $(u-\phi)$ is less than $\alpha$. Since the subsequences converge to $(x_0,t_0)$, we get a contradiction.

We consider two cases: either $x_j=y_j$ infinitely often or $x_j\neq y_j$ for all $j$ large enough. First, let  $x_j=y_j$, and denote
\[
\begin{split}
\vp(y,s)=\frac{j^2}{4}\abs{x_j-y}^4+\frac{j}{2}(t_j-s)^2.
\end{split}
\]
Then
\[
\begin{split}
\phi(y,s)+\vp(y,s)
\end{split}
\]
has a local minimum at $(y_j,s_j)$. Since the function $p$ is continuous, by \eqref{eq:counter-proposition} we have
\[
\begin{split}
(n+p(y_j,s_j))\phi_t(y_j,s_j)-\eta  >  \lambda_\textrm{min}((p(y_j,s_j)-2)D^2\phi(y_j,s_j))+\Delta \phi(y_j,s_j)
\end{split}
\]
for $j$ large enough. As $\phi_t(y_j,s_j)=\vp_t(y_j,s_j)$ and $-D^2
\phi(y_j,s_j)\leq D^2 \vp(y_j,s_j)$, we have by the previous
inequality
\begin{equation}
\label{eq:from-above}
\begin{split}
\eta &< (n+p(y_j,s_j))\vp_t(y_j,s_j)+\lambda_{\textrm{max}}((p(y_j,s_j)-2)D^2\vp(y_j,s_j))+\Delta\vp(y_j,s_j)\\
&=(n+p(x_j,s_j))j(t_j-s_j),
\end{split}
\end{equation}
where we also used the fact that $y_j=x_j$ and thus $D^2\vp(y_j,s_j) =0$.

Next denote
\[
\begin{split}
\psi(x,t)=\frac{j^2}{4}\abs{x-y_j}^4+\frac{j}{2}(t-s_j)^2.
\end{split}
\]
Similarly,
\[
\begin{split}
u(x,t)-\psi(x,t)
\end{split}
\]
has a local maximum at $(x_j,t_j)$, and thus since $D^2\psi(x_j,t_j)=0$, our assumptions imply 
\begin{equation}
\label{eq:from-below}
\begin{split}
0\geq (n+p(x_j,t_j))\psi_t(x_j,t_j)=(n+p(x_j,t_j))j(t_j-s_j),
\end{split}
\end{equation}
for $j$ large enough. This contradicts \eqref{eq:from-above}, because both $t_j$ and $s_j$ converge to $t_0$ and the function $p$ is continuous.

Next we consider the case $y_j\neq x_j$. For the following notation, we refer to \cite{crandallil92} and \cite{juutinenlm01}, \cite{juutinenlp10}. We also use the parabolic theorem of sums for $w_j$
 which implies that there exists symmetric matrices $X_j,Y_j$ such that 
 \[
\begin{split}
&\Big(j(t_j-s_j),\,j^2\abs{x_j-y_j}^2(x_j-y_j),\,X_j\Big)\in  \ol{\mathcal{P}}^{2,+}u(x_j,t_j),\\
&\Big(j(t_j-s_j),\,j^2\abs{x_j-y_j}^2(x_j-y_j),\,Y_j\Big)\in \ol{\mathcal{P}}^{2,-}\phi(y_j,s_j),
\end{split}
\]
and 
\[
\begin{split}
\begin{pmatrix}
X_j&0\\
0&-Y_j 
\end{pmatrix}
\le D^2\Psi_j(x_j,y_j)+\frac1j [D^2\Psi_j(x_j,y_j)]^2
\end{split}
\]
with $\Psi_j(x_j,y_j)=\frac{j^2}{4}\abs{x_j-y_j}^4$.
Here
\[
\begin{split}
D^2\Psi_j(x_j,y_j)=\begin{pmatrix}
M&-M\\
-M&M
\end{pmatrix},
\end{split}
\]
where $M=j^2\abs{x_j-y_j}^2\Big( 2\frac{x_j-y_j}{\abs{x_j-y_j}}\otimes \frac{x_j-y_j}{\abs{x_j-y_j}}+I\Big)$, and
\[
\begin{split}
[D^2\Psi_j(x_j,y_j)]^2=2 \begin{pmatrix}
M^2&-M^2\\
-M^2&M^2 
\end{pmatrix}.
\end{split}
\]
Let $\xi:=\frac{x_j-y_j}{\abs{x_j-y_j}}$ and use $(\sqrt{p(x_j,t_j)-1}\, \xi,\sqrt{p(y_j,s_j)-1}\, \xi)$. The above implies
\begin{align*}
& (p(x_j,t_j)-1) \xi' X_j\cdot \xi-(p(y_j,s_j)-1)\xi' Y_j\cdot \xi\\
&\le C\left(p(x_j,t_j)-p(y_j,s_j)\right)^2\Big(\xi' M\xi+\frac{2}{j} \xi' M^2 \xi\Big),
\end{align*}
where we used an estimate
\[
\left(\sqrt{p(x_j,t_j)-1}-\sqrt{p(y_j,s_j)-1}\right)^2\leq (p(x_j,t_j)-p(y_j,s_j))^2,
\] 
which holds since the function $p$ is greater than 2. 

We have
\[
\begin{split}
\eta &<-(n+p(x_j,t_j))j(t_j-s_j)+(n+p(y_j,s_j))j(t_j-s_j)\\
&\hspace{1 em}+(p(x_j,t_j)-2)\langle X_j \xi,\, \xi\rangle+\tr(X_j)-(p(y_j,s_j)-2)\langle Y_j  \xi,\,\xi\rangle-\tr(Y_j).
\end{split}
\] 
Since the function $p$ is Lipschitz continuous, we have 
\[
\begin{split}
&|-(n+p(x_j,t_j))j(t_j-s_j)+(n+p(y_j,s_j))j(t_j-s_j)|\\
&\hspace{1 em}=|j(t_j-s_j)(p(y_j,s_j)-p(x_j,t_j))|\\
&\hspace{1 em}<C_1j|t_j-s_j|(|x_j-y_j|^2+|t_j-s_j|^2)^{\frac12}\\
&\hspace{1 em}\leq C_1j|t_j-s_j|\sqrt2(|x_j-y_j|+|t_j-s_j|)\\
&\hspace{1 em}=\sqrt2C_1\left((j|t_j-s_j|^2)^{\frac12}(j^2|x_j-y_j|^4)^{\frac14}+j|t_j-s_j|^2\right)\\
&\hspace{1 em}<\frac\eta2
\end{split}
\]
when $j$ is large enough. Hence, we get
\[
\begin{split}
\frac{\eta}{2} &<(p(x_j,t_j)-2)\langle X_j \xi,\, \xi\rangle+\tr(X_j)-(p(y_j,s_j)-2)\langle Y_j  \xi,\,\xi\rangle-\tr(Y_j)\\
&\leq \langle (X_j-Y_j)\xi,\xi\rangle +(p(x_j,t_j)-2)\langle X_j \xi,\, \xi\rangle-(p(y_j,s_j)-2)\langle Y_j  \xi,\,\xi\rangle\\
&\le \langle X_j (\sqrt{p(x_j,t_j)-1}\, \xi),(\sqrt{p(x_j,t_j)-1}\, \xi)\rangle\\
&\hspace{1 em}-\langle Y_j (\sqrt{p(y_j,s_j)-1}\, \xi),(\sqrt{p(y_j,s_j)-1}\, \xi)\rangle\\
&\leq C\left(p(x_j,t_j)-p(y_j,s_j)\right)^2\Big(\xi' M\xi+\frac{2}{j} \xi' M^2 \xi\Big)\\
&\leq C\left(|x_j-y_j|^2+|t_j-s_j|^2\right)(j^2|x_j-y_j|^2+j^3|x_j-y_j|^4)\\
& <C\left(j^2|x_j-y_j|^4+(j^2|x_j-y_j|^4)^{3/2}\right)
\end{split}
\] 
when $j$ is large. This is a contradiction, since $j^2|x_j-y_j|^4\rightarrow 0$ when $j\rightarrow \infty$. In the last two estimates we used Lipschitz continuity of $p$.
\end{proof}

By modifying the above proof we also get the uniqueness. For viscosity solutions we assume continuity on $\overline{\Omega}_T$.

\begin{lemma}
\label{lemma:uniqueness}
Viscosity solutions to \eqref{eq:equation} are unique.
\end{lemma}
\begin{proof}
The proof is by contradiction: We assume that
$u$ and $v$ are viscosity solutions with the same boundary values and yet
\[
\begin{split}
u(x_0,t_0)-v(x_0,t_0)=\sup(u-v)>0.
\end{split}
\]
Further, by considering 
\[
\begin{split}
u-\frac{\eta}{T-t},
\end{split}
\]
we may assume that 
\[
\begin{split}
(n+p(x,t))u_t\leq \Delta_{p(x,t)}^N u-\frac{\eta}{T}
\end{split}
\]
in the viscosity sense when testing from above.

Let
\[
\begin{split}
w_j(x,t,y,s)=u(x,t)-v(y,s)-\Big(\frac{j^2}{4}\abs{x-y}^4+\frac{j}{2}\abs{t-s}^2\Big)
\end{split}
\]
and denote by $(x_j,t_j,y_j,s_j)$ the maximum point of $w_j$ in
$\ol \Om_T\times \ol \Om_T$. Since $(x_0,t_0)$ is a local maximum
for $u-v$, we may assume that
\[
\begin{split}
(x_j,t_j,y_j,s_j)\to (x_0,t_0,x_0,t_0),\quad \textrm{as} \quad j\to \infty
\end{split}
\]
and $(x_j,t_j)\, ,(y_j,s_j)\in \Om_T$.

We consider two cases: either $x_j=y_j$ infinitely often or $x_j\neq y_j$ for all $j$ large enough. First, denote
\[
\begin{split}
\vp(x,t,y,s)=\frac{j^2}{4}\abs{x-y}^4+\frac{j}{2}(t-s)^2
\end{split}
\]
and let  $x_j=y_j$.
Then
$(y,s)\mapsto v(y,s)+\vp(x_j,t_j,y,s)$,
has a local minimum at $(y_j,s_j)$, and $(x,t)\mapsto u(x,t)-\vp(x,t,y_j,s_j)$ a local maximum at $(x_j,t_j)$. From this we deduce  (denote with abuse of notation $\vp(y,s)=\vp(x_j,t_j,y,s)$ in the next display)
\[
\begin{split}
(n+p(y_j,s_j))j(t_j-s_j)=(n+p(y_j,s_j))\vp_s(y_j,s_j)  \ge 0
\end{split}
\]
and (denote with abuse of notation $\vp(x,t)=\vp(x,t,y_j,s_s)$ in the next display)
\[
\begin{split}
(n+p(x_j,t_j))(t_j-s_j)=(n+p(x_j,t_j))\vp_t(x_j,t_j)  \le -\eta/T.
\end{split}
\]
Thus
\[
\begin{split}
\frac{\eta}{T}\le  (n+p(x_j,t_j))(t_j-s_j)-(n+p(y_j,s_j))(t_j-s_j)=0,
\end{split}
\]
a contradiction.

Next we consider the case $y_j\neq x_j$. For the following notation, we refer to \cite{crandallil92} and \cite{juutinenlm01}. We also use the parabolic theorem of sums for $w_j$
 which implies that there exist symmetric matrices $X_j,Y_j$ such that $Y_j-X_j$ is positive semidefinite and
\[
\begin{split}
&\Big(j(t_j-s_j),\,j^2\abs{x_j-y_j}^2(x_j-y_j),\,X_j\Big)\in  \ol{\mathcal{P}}^{2,+}u(y_j,s_j)\\
&\Big(j(t_j-s_j),\,j^2\abs{x_j-y_j}^2(x_j-y_j),\,Y_j\Big)\in \ol{\mathcal{P}}^{2,-}v(x_j,t_j).
\end{split}
\]
Using \eqref{eq:counter-proposition} and the assumptions on $u$, we get
\[
\begin{split}
\frac{\eta}{T}&\le -(n+p(y_j,s_j))j(t_j-s_j)+(n+p(x_j,t_j))j(t_j-s_j)\\
&\hspace{1 em}+ (p(x_j,t_j)-2)\langle Y_j \frac{(x_j-y_j)}{\abs{x_j-y_j}},\, \frac{(x_j-y_j)}{\abs{x_j-y_j}}\rangle+\tr(Y_j)\\
& \hspace{1 em} -(p(y_j,s_j)-2)\langle X_j  \frac{(x_j-y_j)}{\abs{x_j-y_j}},\,\frac{(x_j-y_j)}{\abs{x_j-y_j}}\rangle-\tr(X_j).
\end{split}
\]
The right hand side can be estimated similarly as in the previous lemma to obtain a contradiction.
\end{proof}

\section*{Appendix}
Let us show Cases 1,2, and 3 from the proof of Lemma \ref{pos.laajennus}. Recall that the comparison function in that lemma was
 \[
\Psi(x,t)=\left(\frac19\right)^3\inf_{y\in B_r(0)}u_\eps(y,0)\frac{(\frac13 r)^{2(n+1)^2}}{(t+(\frac13 r)^2)^{(n+1)^2}}\left(9-\frac{|x|^2}{t+(\frac13 r)^2}\right)^2_+
\]
Starting from Case 1, we need to show that for $e\in \R^n$, $|e|=1$, 
\[
\Psi(0,t)\leq \frac12[\Psi(0,t-\frac{\eps^2}{2})+\Psi(\eps e,t-\frac{\eps^2}{2})].
\]
Since
\[
\Psi(0,t)=\frac19 \inf_{y\in B_r(0)} u_\eps(y,0)\frac{(\frac13 r)^{2(n+1)^2}}{[t+(\frac13 r)^2]^{(n+1)^2}},
\]
\[
\Psi(0,t-\frac{\eps^2}{2})=\frac19 \inf_{y\in B_r(0)} u_\eps(y,0)\frac{(\frac13 r)^{2(n+1)^2}}{[t-\frac{\eps^2}{2}+(\frac13 r)^2]^{(n+1)^2}},
\]
and
\[
\Psi(\eps e,t-\frac{\eps^2}{2})=\left(\frac19\right)^3 \inf_{y\in B_r(0)} u_\eps(y,0)\frac{(\frac13 r)^{2(n+1)^2}}{[t-\frac{\eps^2}{2}+(\frac13 r)^2]^{(n+1)^2}}\left[9-\frac{\eps^2}{t-\frac{\eps^2}{2}+(\frac13 r)^2}\right]^2,
\]
we have to show that
\[
1\leq \frac12 \left[\frac{t+(\frac13 r)^2}{t-\frac{\eps^2}{2}+(\frac13 r)^2}\right]^{(n+1)^2}\left[1+\left(\frac19\right)^2\left(9-\frac{\eps^2}{t-\frac{\eps^2}{2}+(\frac13 r)^2}\right)^2\right]=:A_1.
\]
Since
\[
1+\left(\frac19\right)^2\left(9-\frac{\eps^2}{t-\frac{\eps^2}{2}+(\frac13 r)^2}\right)^2\geq 2-\frac29 \left(\frac{\eps^2}{t-\frac{\eps^2}{2}+(\frac13 r)^2}\right),
\]
we get
\begin{align*}
A_1&\geq \left[\frac{t+(\frac13 r)^2}{t-\frac{\eps^2}{2}+(\frac13 r)^2}\right]\left[1-\frac19 \left(\frac{\eps^2}{t-\frac{\eps^2}{2}+(\frac13 r)^2}\right)\right]\\
& =\left[1+\frac12\left(\frac{\eps^2}{t-\frac{\eps^2}{2}+(\frac13 r)^2}\right)\right]\left[1-\frac19 \left(\frac{\eps^2}{t-\frac{\eps^2}{2}+(\frac13 r)^2}\right)\right]\\
& \geq 1-\frac19\left(\frac{\eps^2}{t-\frac{\eps^2}{2}+(\frac13 r)^2}\right)+\frac12 \left(\frac{\eps^2}{t-\frac{\eps^2}{2}+(\frac13 r)^2}\right)-\frac{1}{18}\left(\frac{\eps^2}{t-\frac{\eps^2}{2}+(\frac13 r)^2}\right)\\
& \geq 1,
\end{align*}
and Case 1 is complete.

In Case 2, $|x|=\eta\eps$ for some $0<\eta<1$, and we need to show that
\[
\Psi(x,t)\leq \frac12[\Psi(0,t-\frac{\eps^2}{2})+\Psi(x+\frac{x}{|x|}\eps,t-\frac{\eps^2}{2})].
\]
Since 
\[
\Psi(x,t)=\left(\frac19\right)^3\inf_{y\in B_r(0)} u_\eps(y,0)\frac{(\frac13 r)^{2(n+1)^2}}{[t+(\frac13 r)^2]^{(n+1)^2}}\left[9-\frac{|\eta \eps|^2}{t+(\frac13 r)^2}\right]^2,
\]
\[
\Psi(0,t-\frac{\eps^2}{2})=\frac19\inf_{y\in B_r(0)} u_\eps(y,0)\frac{(\frac13 r)^{2(n+1)^2}}{[t-\frac{\eps^2}{2}+(\frac13 r)^2]^{(n+1)^2}},
\]
and
\[
\Psi(x+\frac{x}{|x|}\eps,t-\frac{\eps^2}{2})=\left(\frac19\right)^3\inf_{y\in B_r(0)} u_\eps(y,0)\frac{(\frac13 r)^{2(n+1)^2}}{[t-\frac{\eps^2}{2}+(\frac13 r)^2]^{(n+1)^2}}\left[9-\frac{|(1+\eta)\eps|^2}{t-\frac{\eps^2}{2}+(\frac13 r)^2}\right]^2,
\]
it is sufficient to show that
\[
\left[9-\frac{|\eta \eps|^2}{t+(\frac{r}{3})^2}\right]^2\leq \frac12 \left(\frac{t+(\frac13 r)^2}{t-\frac{\eps^2}{2}+(\frac13 r)^2}\right)\left[9^2+\left(9-\frac{(1+\eta)^2\eps^2}{t-\frac{\eps^2}{2}+(\frac13 r)^2}\right)^2\right].
\]

Notice that in Cases 1 and 2 we don't need to take the cut-off into account, since $\Psi(x,t)>0$ when $|x|\leq 2\eps$.

Recalling that $r\geq 9\eps$, we have
\[
9^2+\left(9-\frac{(1+\eta)^2\eps^2}{t-\frac{\eps^2}{2}+(\frac13 r)^2}\right)^2\geq 144,
\]
from which it follows that
\begin{align*}
\frac12 &\left(\frac{t+(\frac13 r)^2}{t-\frac{\eps^2}{2}+(\frac13 r)^2}\right)\left[9^2+\left(9-\frac{(1+\eta)^2\eps^2}{t-\frac{\eps^2}{2}+(\frac13 r)^2}\right)^2\right]\\
& =\frac12 \left(1+\frac12\frac{\eps^2}{t-\frac{\eps^2}{2}+(\frac13 r)^2}\right)\left[9^2+\left(9-\frac{(1+\eta)^2\eps^2}{t-\frac{\eps^2}{2}+(\frac13 r)^2}\right)^2\right]\\
& \geq \frac12 \left[9^2+\left(9-\frac{(1+\eta)^2\eps^2}{t-\frac{\eps^2}{2}+(\frac13 r)^2}\right)^2\right]+36 \frac{\eps^2}{t-\frac{\epsilon^2}{2}+(\frac13 r)^2}. 
\end{align*}
Hence, it is enough to show that 
\begin{align}\label{Case2}
&\left[9-\frac{|k\eps|^2}{t+(\frac{r}{3})^2}\right]^2-\frac12\left[9^2+\left(9-\frac{(1+k)^2\eps^2}{t-\frac{\eps^2}{2}+(\frac13 r)^2}\right)^2\right]\nonumber\\
& \leq 36 \frac{\eps^2}{t-\frac{\eps^2}{2}+(\frac13 r)^2}.
\end{align}
The left hand side can be written as
\[
-18\frac{|k\eps|^2}{t+(\frac{r}{3})^2}+\left(\frac{|k\eps|^2}{t+(\frac{r}{3})^2}\right)^2+9\left(\frac{(1+\eta)^2\eps^2}{t-\frac{\eps^2}{2}+(\frac13 r)^2}\right)-\frac12 \left(\frac{(1+\eta)^2\eps^2}{t-\frac{\eps^2}{2}+(\frac13 r)^2}\right)^2. 
\]
Since 
\[
9\left(\frac{(1+\eta)^2\eps^2}{t-\frac{\eps^2}{2}+(\frac13 r)^2}\right)\leq 36 \frac{\eps^2}{t-\frac{\epsilon^2}{2}+(\frac13 r)^2}
\]
and
\[
\left(\frac{|k\eps|^2}{t+(\frac{r}{3})^2}\right)^2\leq \frac{|k\eps|^2}{t+(\frac{r}{3})^2},
\]
inequality \eqref{Case2} holds, and Case 2 is proved.

In Case 3, we need to show that if $|x|\geq \eps$, then
\[
\frac12[\Psi(x+\frac{x}{|x|}\eps,t-\frac{\eps^2}{2})+\Psi(x-\frac{x}{|x|}\eps,t-\frac{\eps^2}{2})]\geq \Psi(x,t).
\]
Suppose first that 
\[
\frac{(|x|+\eps)^2}{t-\frac{\eps^2}{2}+(\frac{r}{3})^2}<9.
\]
Then also 
\[
\frac{|x|^2}{t+(\frac{r}{3})^2}<9\ \text{and}\ \frac{(|x|-\eps)^2}{t-\frac{\eps^2}{2}+(\frac{r}{3})^2}<9.
\]
Since $\left(\frac19\right)^3\inf_{y\in B_r(0)} u_\eps(y,0) \left(\frac{r}{3}\right)^{2(n+1)^2}$ cancels out, it is enough to show that
\begin{align*}
\left(\frac{t-\frac{\eps^2}{2}+(\frac{r}{3})^2}{t+(\frac{r}{3})^2}\right)&\left(9-\frac{|x|^2}{t+(\frac{r}{3})^2}\right)^2\\
& \leq \frac12\left[\left(9-\frac{(|x|+\eps)^2}{t-\frac{\eps^2}{2}+(\frac{r}{3})^2}\right)^2 +\left(9-\frac{(|x|-\eps)^2}{t-\frac{\eps^2}{2}+(\frac{r}{3})^2}\right)^2 \right],
\end{align*}
or equivalently,
\begin{align*}
&\left[9-\frac{|x|^2}{t+(\frac{r}{3})^2}\right]^2\\
& \leq \frac12 \left(1+\frac12 \frac{\eps^2}{t-\frac{\eps^2}{2}+(\frac{r}{3})^2}\right)\left[\left(9-\frac{(|x|+\eps)^2}{t-\frac{\eps^2}{2}+(\frac{r}{3})^2}\right)^2+\left(9-\frac{(|x|-\eps)^2}{t-\frac{\eps^2}{2}+(\frac{r}{3})^2}\right)^2\right].
\end{align*}

This is equivalent to showing that
\begin{align*}
18\left(\frac{\eps^2}{t-\frac{\eps^2}{2}+(\frac13 r)^2}\right)&\leq \frac12 \left(\frac{(|x|+\eps)^4+(|x|-\eps)^4}{[t-\frac{\eps^2}{2}+(\frac13 r)^2]^2}\right)-\frac{|x|^4}{[t+(\frac13 r)^2]^2}\\
& +\frac14 \left(\frac{\eps^2}{t-\frac{\eps^2}{2}+(\frac13 r)^2}\right)I,
\end{align*}
where
\[
I=\left(9-\frac{(|x|+\eps)^2}{t-\frac{\eps^2}{2}+(\frac{r}{3})^2}\right)^2+\left(9-\frac{(|x|-\eps)^2}{t-\frac{\eps^2}{2}+(\frac{r}{3})^2}\right)^2.
\]
Since
\begin{align*}
\frac{|x|^4}{[t+(\frac13 r)^2]^2}&=\left(1-\frac12 \frac{\eps^2}{t+(\frac r3)^2}\right)^2\left(\frac{|x|^4}{[t-\frac{\eps^2}{2}+(\frac13 r)^2]^2}\right)\\
&\leq \left(1-\frac12 \frac{\eps^2}{t+(\frac r3)^2}\right)\left(\frac{|x|^4}{[t-\frac{\eps^2}{2}+(\frac13 r)^2]^2}\right)
\end{align*}
and 
\[
(|x|+\eps)^4+(|x|-\eps)^4\leq 2|x|^4+12|x|^2\eps^2,
\]
we get an estimate
\begin{align*}
\frac12& \left(\frac{(|x|+\eps)^4+(|x|-\eps)^4}{[t-\frac{\eps^2}{2}+(\frac13 r)^2]^2}\right)-\frac{|x|^4}{[t+(\frac13 r)^2]^2}\\
& \geq \frac12 \left(\frac{(|x|+\eps)^4+(|x|-\eps)^4}{[t-\frac{\eps^2}{2}+(\frac13 r)^2]^2}\right)-\left(1-\frac12 \frac{\eps^2}{t+(\frac r3)^2}\right)\left(\frac{|x|^4}{[t-\frac{\eps^2}{2}+(\frac13 r)^2]^2}\right)\\
& \geq 6\left(\frac{|x|^2}{t-\frac{\eps^2}{2}+(\frac13 r)^2}\right)\left(\frac{\eps^2}{t-\frac{\eps^2}{2}+(\frac13 r)^2}\right)+\frac12 \left(\frac{|x|^2}{t-\frac{\eps^2}{2}+(\frac13 r)^2}\right)^2\left(\frac{\eps^2}{t-\frac{\eps^2}{2}+(\frac13 r)^2}\right).
\end{align*}
Hence, it is sufficient to show that
\begin{align*}
18&\left(\frac{\eps^2}{t-\frac{\eps^2}{2}+(\frac13 r)^2}\right)\\
& \leq 6\left(\frac{|x|^2}{t-\frac{\eps^2}{2}+(\frac13 r)^2}\right)\left(\frac{\eps^2}{t-\frac{\eps^2}{2}+(\frac13 r)^2}\right)+\frac12\left(\frac{|x|^2}{t-\frac{\eps^2}{2}+(\frac13 r)^2}\right)^2\left(\frac{\eps^2}{t-\frac{\eps^2}{2}+(\frac13 r)^2}\right)\\
& + \frac14 \left(\frac{\eps^2}{t-\frac{\eps^2}{2}+(\frac13 r)^2}\right)I.
\end{align*}
If
\[
\frac{|x|^2}{t-\frac{\eps^2}{2}+(\frac13 r)^2}\geq \frac52,
\]
the previous inequality clearly holds. If
\[
\frac{|x|^2}{t-\frac{\eps^2}{2}+(\frac13 r)^2}\leq \frac52,
\]
then $I\geq 72$ and the previous inequality holds again.

When
\[
\frac{(|x|+\eps)^2}{t-\frac{\eps^2}{2}+(\frac{r}{3})^2}\geq 9,
\]
we need to show that
\[
\left(9-\frac{|x|^2}{t+(\frac{r}{3})^2}\right)^2_+\leq \frac12 \left(1+\frac12 \frac{\eps^2}{t-\frac{\eps^2}{2}+(\frac{r}{3})^2}\right)\left(9-\frac{(|x|-\eps)^2}{t-\frac{\eps^2}{2}+(\frac{r}{3})^2}\right)^2_+,
\]
and this follows by the previous estimates of Case 3.

Let us next show that $\Psi$ is a viscosity subsolution to
\[
(n+2)u_t(x,t)=\Delta u(x,t).
\]
Denote
\[
|z|^2=\frac{|x|^2}{t+r^2}.
\]
Then
\begin{align*}
&(n+2)\Psi_t(x,t)-\Delta \Psi(x,t)\\
=& \frac{(\frac r3)^{2(n+1)^2}}{(t+(\frac r3)^2)^{(n+1)^2+1}}(9-|z|^2)_+\\
& \left[-(n+2)(n+1)^2(9-|z|^2)_+ +2(n+2)|z|^2+4n-\frac{8|z|^2}{(9-|z|^2)_+}\right]\\
& =:\frac{(\frac r3)^{2(n+1)^2}}{(t+(\frac r3)^2)^{(n+1)^2+1}}(9-|z|^2)_+ A.
\end{align*}
When $a:=9-|z|^2>0$, we get
\begin{align*}
aA&=-(n+2)(n+1)^2a^2 +2(n+2)(9-a)a+4na-8(9-a)\\
&=\left[-(n+2)(n+1)^2-2(n+2)\right]a^2+22(n+2)a-72<0
\end{align*}
when $a=8$, and the discriminant is
\[
D=22^2(n+2)^2-4\times 72(n+2)[(n+1)^2+2]<0,
\]
since $(n+1)^2+2>2(n+2)$. Hence $A<0$ when $0<a\leq 9$.

That $\Psi$ is a subsolution in $\Omega_T$ follows from the fact that the maximum of two subsolutions is a subsolution.


\end{document}